\documentclass[hidelinks,onefignum,onetabnum]{siamart250211}

\usepackage{amssymb}
\usepackage{amsfonts}
\usepackage{algpseudocode}
\usepackage{bm}
\usepackage{booktabs}
\usepackage{graphicx}
\usepackage{array}

\newsiamremark{remark}{Remark}
\newsiamthm{assumption}{Assumption}

\newcommand{\R}{\mathbb{R}}
\newcommand{\E}{\mathbb{E}}
\newcommand{\rd}{\mathrm{d}}
\newcommand{\AAm}{\bm{A}}
\DeclareMathOperator*{\argmin}{arg\,min}
\newcolumntype{L}[1]{>{\raggedright\arraybackslash}p{#1}}

\headers{Structure-Adaptive RFM for Elliptic PDEs}{J.~Linghu, H.~Dong, and Y.~Wang}

\title{A Structure-Adaptive Random Feature Method for High-Dimensional Elliptic PDEs}

\author{Jiale Linghu\thanks{School of Mathematics and Statistics, Xidian University,
Xi'an 710071, China.}
\and Hao Dong\footnotemark[1]
\and Yangshuai Wang\thanks{Department of Mathematics, National University of Singapore,
10 Lower Kent Ridge Road, 119076, Singapore
(\email{yswang@nus.edu.sg}, corresponding author).}}

\ifpdf
\hypersetup{
  pdftitle={A Structure-Adaptive Random Feature Method for High-Dimensional Elliptic PDEs},
  pdfauthor={J. Linghu, H. Dong, Y. Wang}
}
\fi

\begin{document}

\maketitle

\begin{abstract}
Random-feature methods reduce high-dimensional elliptic PDE collocation to linear coefficient problems, but full-dimensional trial spaces overlook lower-dimensional structure.
We introduce the Hierarchical Analysis-of-Variance Random Feature Method (HA-RFM), which selects coordinate blocks using closed Sobol indices of the PDE residual, identifies oblique low-rank features from fitted-predictor gradients, and couples all retained features in one regularized least-squares solve.
Under structural and stability hypotheses, we establish an $L^2$ error bound that links solution and residual truncation to finite-width approximation and regularized finite-sample fitting, and we derive guarantees for width and structure recovery.
The resulting width is polynomial in the dimension at fixed interaction order, with dimension-independent higher-order contributions under uniform structural control.
Residual screening achieves exact recovery of the prescribed three-pair support, while fitted-predictor gradients recover oblique directions through dimension~$50$.
In random-ridge tests, less than $1\%$ additional width reduces errors by factors of $14$--$39$ over coordinate blocks and $34$--$100$ over equal-width full-dimensional RFM.
Semilinear computations extend HA-RFM through dimension $100$, while dense and distributed interactions delineate the coordinate families required for broader structure.
\end{abstract}



\section{Introduction}
\label{sec:introduction}

High-dimensional elliptic PDEs arise in uncertainty quantification~\cite{cohen2015parametric}, stochastic control and Hamilton--Jacobi--Bellman equations~\cite{darbon2016algorithms,nakamura2021adaptive}, and kinetic or Fokker--Planck models~\cite{risken1989fokker}.
Their solutions often concentrate in low-order coordinate interactions, a low-dimensional oblique variable, or both.
A practical solver should therefore allocate trial-space width to this effective structure rather than uniformly to generic functions of all ambient variables.

Existing approaches exploit complementary forms of structure.
Sparse grids and dimension-adaptive variants use mixed regularity or influential coordinates~\cite{smolyak1963,griebel2010,gerstner2003dimadaptive,schillings2013sparse}, while tensor trains~\cite{oseledets2011tt} and sparse-polynomial Galerkin methods~\cite{cohen2011sparsegalerkin,cohen2015parametric} exploit low rank, anisotropy, or parametric regularity.
Deep PDE solvers avoid explicit tensor-product grids but generally require high-dimensional nonconvex training~\cite{raissi2019pinn,e2018deepritz,e2017deepbsde,wang2022pinn-ntk}.
Random feature methods (RFM)~\cite{rahimi2007random,rahimi2008uniform,bach2017kernelquadrature,rudi2017generalization} instead reduce PDE collocation to a linear least-squares fit once the nonlinear features are sampled~\cite{huang2006elm,chen2022rfm,wang2024elmhd}.
Full-dimensional RFM nevertheless assigns every feature to the ambient variables, leaving its approximation cost tied to full-domain function complexity~\cite{bach2017kernelquadrature,ema2019barron}.
What is missing is a computable, PDE-driven rule that converts detected structure into trial-space geometry while retaining the frozen-feature linear solve.

Analysis-of-variance (ANOVA) decompositions and Sobol indices describe interactions aligned with the coordinate axes~\cite{hoeffding1948,rabitz2000hdmr,sobol2001}.
Active subspaces describe oblique low-rank dependence~\cite{constantine2015active,constantine2014active,lam2020multifidelity,zahm2020gradient}.
We combine these descriptors in the Hierarchical Analysis-of-Variance Random Feature Method (HA-RFM).
Closed Sobol indices of the current PDE residual select higher-order coordinate blocks; the empirical covariance of the fitted-predictor gradient identifies oblique directions through prescribed spectral-energy and eigengap criteria.
All retained coordinate and active-subspace features are then fitted in one regularized least-squares problem, preserving the coupling imposed by the differential operator and boundary condition.
For semilinear equations, the same construction supplies the elliptic solve at each Picard step.

The paper makes three contributions.
First, we construct a PDE-driven trial space in which residual Sobol scores allocate coordinate blocks, predictor-gradient covariance allocates oblique features, and one joint linear solve couples all retained blocks.
Second, under explicit structural, elliptic-stability, block-stability, and sampled-least-squares hypotheses, we establish an error estimate for the HA-RFM approximation that separates structural truncation, random-feature approximation, sampling, and regularization.
Its consequences include polynomial width at fixed interaction order, a dimension-independent higher-order contribution when the selected family, aggregate component amplitudes, and stability factors are uniformly bounded, and estimates for residual-Sobol recovery, predictor-based subspace identification, coordinate--ridge augmentation, and inexact Picard iteration.
Third, we validate the two allocation mechanisms numerically.
Residual screening with two quasi-Monte Carlo (QMC) base matrices of size $4096$ recovers a three-pair support for all nine combinations of feature and QMC seeds, while predictor gradients recover oblique directions through dimension $50$.
In random-ridge tests, a below-$1\%$ width increase reduces the coordinate-block error by factors of $14$--$39$ and the equal-width full-dimensional RFM error by factors of $34$--$100$.
Semilinear computations demonstrate the Picard extension through dimension $100$, while dense and distributed-interaction tests identify when broader coordinate families are required.

Section~\ref{sec:background} introduces the RFM, ANOVA--Sobol, and active-subspace ingredients used to construct HA-RFM in Section~\ref{sec:method}.
Section~\ref{sec:analysis} follows the algorithmic dependency order from error and width guarantees to screening, subspace, and Picard estimates; complete proofs are provided in the supplement.
Section~\ref{sec:experiments} tests these mechanisms across coordinate-aligned, oblique, semilinear, and distributed-interaction problems, and Section~\ref{sec:conclusion} concludes.

\smallskip
\noindent\textbf{Notation.}
Unless otherwise stated, $d$ denotes the ambient dimension, $\Omega \subset \mathbb{R}^d$ the computational domain, and $x \in \Omega$ a point.
Set $\mathcal I_d:=\{1,\ldots,d\}$.
For product-domain statements, $\Omega=\prod_{i=1}^d\Omega_i$, where $\Omega_i\subset\mathbb R$ is the $i$th coordinate domain.
For $S \subseteq \mathcal I_d$, set $\Omega_S=\prod_{i\in S}\Omega_i$, and let $|S|$ and $x_S$ denote its cardinality and coordinate subvector; an ANOVA block indexed by $S$ depends only on $x_S$.
Calligraphic letters such as $\mathcal A$ denote families of coordinate subsets $S\subseteq\mathcal I_d$, and $\mathcal V_{S,M}$ denotes the span of $M$ random features supported on the coordinates in $S$, so $\dim\mathcal V_{S,M}\le M$.
The letter $C$, with descriptive subscripts when needed, denotes a positive scalar constant whose dependencies are stated.

\section{Background}
\label{sec:background}

We recall the notation for random-feature least-squares collocation, ANOVA--Sobol descriptions of coordinate-aligned interactions, and active subspaces for oblique low-rank dependence.
These ingredients enter the HA-RFM trial space constructed in Section~\ref{sec:method}.

\subsection{Random feature method for PDEs}
\label{subsec:bg-rfm}

For a positive integer $M$, fix an activation $\sigma$ and sample frequencies $\omega_m\in\mathbb R^d$ and biases $b_m\in\mathbb R$ once.
The resulting RFM trial space~\cite{rahimi2007random,rahimi2008uniform} is
\[
\mathcal V_M
=
\operatorname{span}\{\phi_m(x)=\sigma(\omega_m\cdot x+b_m):1\le m\le M\}.
\]
Only the coefficients $c_m\in\mathbb R$ in $u_M=\sum_{m=1}^M c_m\phi_m$ are fitted.
For an elliptic operator $L$ with right-hand side $f$, strong-form interior rows enforce $Lu_M(x_n^{\rm int})=f(x_n^{\rm int})$ at collocation points $x_n^{\rm int}\in\Omega$; for Dirichlet data $g$, boundary rows enforce $u_M(x_n^{\rm bc})=g(x_n^{\rm bc})$ at $x_n^{\rm bc}\in\partial\Omega$.
With $N$ collocation conditions, let $A\in\mathbb R^{N\times M}$ be the matrix whose rows evaluate either $L\phi_m$ or the boundary trace of $\phi_m$, let $y\in\mathbb R^N$ contain the corresponding data, and fix $\lambda_{\rm reg}\ge0$.
A schematic Tikhonov-regularized fit is
\[
c_{\lambda_{\rm reg}}
\in
\argmin_{c\in\R^M}
\|A c-y\|_2^2+\lambda_{\rm reg}\|c\|_2^2.
\]
The frozen nonlinear parameters therefore reduce the PDE solve to a linear coefficient problem~\cite{huang2006elm,chen2022rfm,wang2024elmhd}; for strong-form second-order equations, $\sigma$ is chosen with the required differentiability.

For a coordinate subset $S$ of size $k$ and $1\le m\le M$, define
$\phi_{S,m}(x)=\sigma(\omega_{S,m}\cdot x_S+b_{S,m})$, where $\omega_{S,m}\in\mathbb R^k$ and $b_{S,m}\in\mathbb R$.
This feature varies only in $x_S$ but remains a function on $\Omega$.
Sums of these subset-supported spaces retain the linear coefficient fit and make each blockwise approximation problem $k$-variate.
The analysis assumes $H^2(\Omega_S)$ approximation of each selected component and, for the width estimate, the model rate $M^{-s_{\rm rf}/(2k)}$ with random-feature smoothness $s_{\rm rf}>0$, up to the stated logarithmic and frequency-truncation factors~\cite{bach2017kernelquadrature,rudi2017generalization,ema2019barron}.
Thus the per-block exponent depends on $|S|$, whereas a full-dimensional RFM has $k=d$.

\subsection{Integral ANOVA decomposition and Sobol indices}
\label{subsec:bg-anova}

Let each $\mu_i$ be a probability measure on $\Omega_i$, and set $\mu=\otimes_{i=1}^d\mu_i$, the product measure on $\Omega=\prod_{i=1}^d\Omega_i$ used for ANOVA inner products; in the cube examples, $\mu$ is normalized Lebesgue measure.
All expectations and variances below are taken with respect to $\mu$.
We use $L^2_\mu(\Omega)$ for the ANOVA norm and retain $L^2(\Omega)$ for the unnormalized PDE norm; these norms agree on $(0,1)^d$ and differ only by a constant normalization on the other cube domains.
For $u\in L^2_\mu(\Omega)$, the integral ANOVA decomposition is~\cite{hoeffding1948,rabitz2000hdmr,caflisch1998}
\begin{equation}
\label{eq:anova-bg}
u(x) \;=\; \sum_{S\subseteq \mathcal I_d} u_S(x_S),
\end{equation}
where $u_\emptyset=\int_\Omega u\,\rd\mu$ and, for $S\ne\emptyset$,
$u_S=\mathbb E[u\,|\,x_S]-\sum_{T\subsetneq S}u_T$.
Each nonempty component has zero mean in every coordinate indexed by $S$; these constraints make the decomposition unique and imply $L^2_\mu$ orthogonality between distinct components.

For $\operatorname{Var}(u)>0$, the pure and closed Sobol indices are, respectively,
\[
\mathsf S_S
=
\frac{\|u_S\|_{L^2_\mu(\Omega)}^2}{\sum_{\emptyset\ne T\subseteq\mathcal I_d}\|u_T\|_{L^2_\mu(\Omega)}^2},
\qquad
T_S^{\rm closed}
=
\frac{\operatorname{Var}\bigl(\mathbb E[u\,|\,x_S]\bigr)}{\operatorname{Var}(u)}
=
\sum_{\emptyset\ne T\subseteq S} \mathsf S_T.
\]
The closed index is estimated directly by Saltelli pick-freeze estimators~\cite{sobol2001,saltelli2010,owen2014sobol}, and pure interactions can be recovered by M\"obius inversion when needed.
HA-RFM applies the closed indices to the current PDE residual to select coordinate blocks; Corollary~\ref{cor:screening-complexity} gives the gap and perturbation conditions for recovering an ideal correction family.

For $u\not\equiv0$ and a relative tail tolerance $\tau\ge0$, define
\[
K_\tau
=
\min\left\{K\in\{0,\ldots,d\}:\,
\left\|u-\sum_{|S|\le K}u_S\right\|_{L^2_\mu(\Omega)}
\le \tau\|u\|_{L^2_\mu(\Omega)}
\right\}.
\]
This superposition dimension is small when the $L^2_\mu$ tail above a fixed interaction order is small; for fixed $K_\tau$, the number of candidate blocks grows polynomially in $d$~\cite{rabitz2000hdmr,kuo2010}.
The analysis uses ANOVA orthogonality to separate truncation from per-block approximation, while a block-stability assumption controls the nonorthogonality of the raw feature spaces used computationally.

\subsection{Active subspaces}
\label{subsec:bg-as}

ANOVA structure is coordinate aligned, whereas active subspaces describe oblique low-dimensional dependence; the subscript ${\rm AS}$ marks active-subspace quantities.
For a scalar signal $q:\Omega\to\mathbb R$ with weak gradient $\nabla q\in L^2(\Omega,\mu;\mathbb R^d)$, define the covariance functional
\[
\mathbf C[q]:=\int_\Omega \nabla q(x)\nabla q(x)^\top\,\rd\mu(x)
\in\mathbb R^{d\times d}.
\]
When $q=u$, we write $\mathbf C_u:=\mathbf C[u]$.
Let $\nu_1\ge\nu_2\ge\cdots\ge\nu_d\ge0$ be the eigenvalues of $\mathbf C[q]$.
Its leading eigenspace contains the directions of strongest mean-square variation; a gap $\nu_{r_{\rm AS}}>\nu_{r_{\rm AS}+1}$, with $1\le r_{\rm AS}<d$, separates an $r_{\rm AS}$-dimensional active subspace~\cite{constantine2015active}.
Let $W\in\mathbb R^{d\times r_{\rm AS}}$ have orthonormal columns spanning this population active subspace.
When $q=u$ admits the ridge representation $u(x)=h(V^\top x)$ with an orthonormal matrix $V\in\mathbb R^{d\times r_{\rm AS}}$ and profile $h:\mathbb R^{r_{\rm AS}}\to\mathbb R$, write $y=V^\top x$ for the reduced variable and $\nabla_y h$ for the gradient of the profile.
The covariance then has the form $\mathbf C_u=V\mathbf B_{\rm red}V^\top$, where
\[
\mathbf B_{\rm red}:=\int_\Omega \nabla_y h(V^\top x)\nabla_y h(V^\top x)^\top\,\rd\mu(x)
\in\mathbb R^{r_{\rm AS}\times r_{\rm AS}}.
\]
Hence $\operatorname{range}(\mathbf C_u)\subseteq \operatorname{range}(V)$, with equality when $\mathbf B_{\rm red}$ is nonsingular; in that case, $V$ and $W$ span the same subspace.

Suppose $\mathbf C[q]$ has eigengap $\Delta=\nu_{r_{\rm AS}}-\nu_{r_{\rm AS}+1}>0$ and $\widehat{\mathbf C}$ is an empirical approximation with leading basis $\widehat W_{r_{\rm AS}}\in\mathbb R^{d\times r_{\rm AS}}$.
Writing $\Theta(\widehat W_{r_{\rm AS}},W)$ for the principal angles and $\|\cdot\|_{\rm op}$ for the operator norm, the Davis--Kahan bound~\cite{davis-kahan1970} gives, for a universal constant $C_{\rm DK}$,
\[
\|\sin\Theta(\widehat W_{r_{\rm AS}},W)\|_{\rm op}
\le
C_{\rm DK}\|\widehat{\mathbf C}-\mathbf C[q]\|_{\rm op}/\Delta.
\]
HA-RFM forms $\widehat{\mathbf C}$ from fitted-predictor gradients, activates a low-rank block only when prescribed energy and eigengap tests are met, and samples that block within the estimated eigenspace.
Section~\ref{subsec:method-active-subspace} makes this construction explicit, and Section~\ref{subsec:as-analysis} controls its subspace error by the perturbation bound above.

\section{Hierarchical ANOVA-RFM}
\label{sec:method}

HA-RFM allocates random features in two stages: residual Sobol indices select coordinate blocks, and fitted-predictor gradients identify an oblique low-rank block when the prescribed spectral tests are met.
All retained features are fitted jointly with frozen random parameters; we present this construction in computational order and then extend it to Picard iteration.

\subsection{Problem and coordinate-block trial space}
\label{subsec:method-setting}

We describe the construction for the unknown $u$ in the reaction--diffusion problem
\begin{equation}
\label{eq:pde}
-\Delta u + \kappa(x)\, u = f(x) \quad \text{in } \Omega, \qquad u = g \quad \text{on } \partial \Omega,
\end{equation}
on $\Omega=(0,1)^d$, and write
$L u := -\Delta u+\kappa u$.
Here $\kappa \in L^\infty(\Omega)$, $f$ is the source, and $g$ is the Dirichlet datum.
The analysis covers uniformly elliptic principal parts and bounded lower-order terms under the stability hypotheses of Section~\ref{sec:analysis}, and the experiments also use $(-1,1)^d$.
We use strong-form collocation with sufficiently differentiable data and features; boundary rows impose nonhomogeneous data, while the analysis applies the residual estimate after a fixed boundary lift.
Other bounded linear lower-order terms only change the operator evaluations in the collocation rows.

We use the ANOVA notation from Section~\ref{subsec:bg-anova} to organize the selected coordinate subsets.
For $S\subseteq\mathcal I_d$ with $|S|=k$, the subset-supported random features are
\begin{equation}
\label{eq:basis}
\phi_{S,m}(x)
=
\sigma\bigl(\omega_{S,m}\cdot x_S+b_{S,m}\bigr),
\qquad
m=1,\dots,M_k,
\end{equation}
where $M_k$ is the number of features assigned to each selected order-$k$ subset, $\sigma\in C^2(\R)$ is fixed, $\omega_{S,m}\in\R^k$, and $b_{S,m}\in\R$.
Section~\ref{sec:experiments} specifies their sampling distributions.
The empty subset is always included with $M_0=1$ and $\phi_{\emptyset,1}\equiv 1$.
For a maximum interaction order $1\le K_{\max}\le d$, the selected families are initialized and updated by
\[
\mathcal A_1=\{\emptyset\}\cup\{\{i\}:1\le i\le d\},
\qquad
\mathcal A_{K+1}=\mathcal A_K\cup\mathcal S_{K+1},
\quad 1\le K<K_{\max},
\]
where $\mathcal S_{K+1}$ is selected by the residual screen below.
For real coefficients $\bm c=\{c_{S,m}:S\in\mathcal A_K,\ 1\le m\le M_{|S|}\}$, the order-$K$ HA-RFM trial function is
\begin{equation}
\label{eq:ansatz}
v_{\bm c}^{(K)}(x)
=
\sum_{S\in\mathcal A_K}
\sum_{m=1}^{M_{|S|}} c_{S,m}\,\phi_{S,m}(x).
\end{equation}
The constant contributes one degree of freedom, absorbed in the width estimates, and every other basis function varies only in $x_S$.
For the operator in~\eqref{eq:pde},
\[
\Delta \phi_{S,m}(x)
=
\left(\sum_{j\in S}\omega_{S,m,j}^2\right)
\sigma''(\omega_{S,m}\cdot x_S+b_{S,m}).
\]
Here $\omega_{S,m,j}$ denotes the frequency component multiplying $x_j$; derivatives in the inactive coordinates vanish.
Thus the PDE rows can be assembled without ever forming full $d$-variate random frequencies for a coordinate block.

The ANOVA subsets index coordinate structure, but the numerical blocks are raw and uncentered, so their spans may overlap.
The joint Tikhonov fit below regularizes this representation, and the analysis controls the summed trial function rather than individual fitted blocks.

\subsection{Residual screening}
\label{subsec:method-screening}

The family is enriched from the constant and singleton blocks; retaining all singletons produces the $O(dM_1)$ width term.
At order $K$, let $\widehat{\bm c}_K$ be the coefficients returned by the joint solve in Section~\ref{subsec:method-ls} and define the fitted coordinate predictor $\widehat u_K:=v_{\widehat{\bm c}_K}^{(K)}$.
Candidate subsets of order $K+1$ are screened using the residual
\begin{equation}
\label{eq:res}
\mathfrak r_K(x)
:=
f(x)-L\widehat u_K(x).
\end{equation}
Prescribed nonnegative held-out functionals $\mathfrak E_{\rm int}$ and $\mathfrak E_{\rm bc}$ measure the interior and boundary residuals on the current trial space.
For tolerances $\varepsilon_{\rm int}^{\rm stop},\varepsilon_{\rm bc}^{\rm stop}\ge0$, enrichment stops when both tests meet their respective tolerances.
For each candidate $S$ with $|S|=K+1$, HA-RFM estimates the closed residual Sobol index.
Let $X\sim\mu$ denote the reference random input on $\Omega$ and let $X_S$ be its coordinate subvector.
When $\operatorname{Var}(\mathfrak r_K)>0$, this index is
\begin{equation}
\label{eq:sobol-closed}
T_S^{\rm closed}(\mathfrak r_K)
=
\frac{
\operatorname{Var}\bigl(\E[\mathfrak r_K\,|\,X_S]\bigr)
}{
\operatorname{Var}(\mathfrak r_K)
}.
\end{equation}
The estimate $\widehat T_S^{\rm closed}(\mathfrak r_K)$ uses a Saltelli pick-freeze construction~\cite{sobol2001,saltelli2010,owen2014sobol} with two randomized quasi-Monte Carlo (QMC) base matrices of size $M_{\rm QMC}$~\cite{owen1997scrambled} and the corresponding $S$-hybrid matrices.
Scores are truncated below at zero; if the empirical residual variance or all truncated scores vanish, every screening score is set to zero.
Fix the relative screening threshold $\vartheta_{\rm Sob}\in(0,1)$.
The selected next-level family is
\begin{equation}
\label{eq:active-set}
\mathcal S_{K+1}
=
\left\{
S:\ |S|=K+1,\;
\widehat T_S^{\rm closed}(\mathfrak r_K)
>
\vartheta_{\rm Sob}\,\max_{|S'|=K+1}\widehat T_{S'}^{\rm closed}(\mathfrak r_K)
\right\}.
\end{equation}
With a candidate cap, only the highest-scoring candidates up to that cap are retained; otherwise~\eqref{eq:active-set} is used unchanged.
The update $\mathcal A_{K+1}=\mathcal A_K\cup\mathcal S_{K+1}$ stops if $\mathcal S_{K+1}=\emptyset$, so the search targets hierarchically detectable interactions and does not proceed to order $K+2$ after an empty level.
Full order-$k$ screening costs $O(M_{\rm QMC}\binom{d}{k})$ residual evaluations, up to shared base evaluations, independently of the fitted trial-space width.
Section~\ref{sec:analysis} gives the closed-index gap and perturbation conditions for correct selection.

\subsection{Joint coefficient fitting}
\label{subsec:method-ls}
For each selected family, a single least-squares (LS) problem fits all block coefficients.
Let $\{x_n^{\rm int}\}_{n=1}^{N_{\rm int}}\subset\Omega$ be interior collocation points and let $\{x_n^{\rm bc}\}_{n=1}^{N_{\rm bc}}\subset\partial\Omega$ be boundary samples.
Let $\mathcal D$ denote the current frozen feature family and set $N_{\mathcal D}:=|\mathcal D|$.
Before the active-subspace augmentation, $\mathcal D=\{\phi_{S,m}:S\in\mathcal A_K,\ 1\le m\le M_{|S|}\}$; after that augmentation, the global active-subspace features are added to the same feature family.
For columns indexed by $\psi_j\in\mathcal D$, define $\AAm_{\rm int}\in\R^{N_{\rm int}\times N_{\mathcal D}}$ and $\AAm_{\rm bc}\in\R^{N_{\rm bc}\times N_{\mathcal D}}$ by
\[
(\AAm_{\rm int})_{n,j}
=
(L\psi_j)(x_n^{\rm int}),
\qquad
(\AAm_{\rm bc})_{n,j}
=
\psi_j(x_n^{\rm bc}).
\]
The right-hand sides $\bm f\in\R^{N_{\rm int}}$ and $\bm g\in\R^{N_{\rm bc}}$ have entries $(\bm f)_n=f(x_n^{\rm int})$ and $(\bm g)_n=g(x_n^{\rm bc})$.
Let $\omega_{\rm bc}>0$ be the boundary penalty and $\lambda_{\rm reg}\ge0$ the Tikhonov parameter.
The coefficient vector is obtained from
\begin{equation}
\label{eq:joint-ls}
\widehat{\bm c}
\in
\argmin_{\bm c\in\R^{N_{\mathcal D}}}
\frac{1}{N_{\rm int}}\|\AAm_{\rm int}\bm c-\bm f\|_2^2
+\omega_{\rm bc}^2\frac{1}{N_{\rm bc}}\|\AAm_{\rm bc}\bm c-\bm g\|_2^2
+\lambda_{\rm reg}\|\bm c\|_2^2.
\end{equation}
The associated fitted function is
\[
\widehat u_{\mathcal D}(x)=\sum_{j=1}^{N_{\mathcal D}}\widehat c_j\psi_j(x).
\]
For the coordinate-only family through order $K$, this function is the predictor $\widehat u_K$ defined above.
The normalization sets the scaling of the empirical Gram matrices used in the analysis.
Equivalently, QR with column pivoting is applied in the least-squares sense to the unnormalized augmented system
\[
[\AAm_{\rm int};\,\omega_{\rm bc,impl}\AAm_{\rm bc};\,\sqrt{\lambda_{\rm impl}} I]\bm c
\approx
[\bm f;\,\omega_{\rm bc,impl}\bm g;\,0],
\]
where $I$ is the $N_{\mathcal D}\times N_{\mathcal D}$ identity matrix, $\omega_{\rm bc,impl}=\omega_{\rm bc}\sqrt{N_{\rm int}/N_{\rm bc}}$, and $\lambda_{\rm impl}=N_{\rm int}\lambda_{\rm reg}$.
Tikhonov regularization controls cross-block correlations, while pivoted QR supplies a rank-revealing solve.

\paragraph{Why the fit is joint}
\label{subsec:method-why-joint}
Boundary rows couple the selected blocks: under exact homogeneous Dirichlet enforcement, the constraint
\begin{equation}
\label{eq:bc-sum}
\sum_{S\in\mathcal A_K}v_S(x_S)=0
\qquad \text{on }\partial\Omega
\end{equation}
acts on their sum, where $v_S$ denotes the raw contribution of block $S$, not an orthogonal ANOVA component.
A sequential level-one solve would impose this condition using only a constant and univariate blocks, which collapses the available space as follows.

\begin{lemma}[univariate-block collapse on the cube]
\label{lem:bc-collapse}
Let $\Omega=(0,1)^d$ with $d\ge2$, let $c\in\R$, and let $v_1,\dots,v_d\in C([0,1])$ satisfy
\begin{equation}
\label{eq:lem-faces}
c+\sum_{i=1}^d v_i(x_i)=0
\qquad \text{on every face of }\partial\Omega .
\end{equation}
Then each $v_i$ is constant on $[0,1]$ and $c+\sum_i v_i\equiv0$.
If, in addition, each $v_i$ has zero mean on $[0,1]$, then $v_i\equiv0$ for every $i$ and $c=0$.
\end{lemma}

\begin{proof}
On the face $\{x_k=0\}$, $\sum_{j\ne k}v_j(x_j)=-c-v_k(0)$; varying any $x_j$ with $j\ne k$ shows that $v_j$ is constant.
Choosing $k\ne j$ for each $j$ and substituting proves $c+\sum_i v_i\equiv0$.
The zero-mean conditions then force every $v_i$, and consequently $c$, to vanish.
\end{proof}

Thus levelwise boundary enforcement can overconstrain lower-order blocks before higher-order cancellation is available, whereas~\eqref{eq:joint-ls} lets all retained blocks share the interior and boundary residuals.

\subsection{Predictor covariance and active block}
\label{subsec:method-active-subspace}

Coordinate-aligned ANOVA blocks can require many selected interactions for a function of the form $u(x)=h(W^\top x)$ when the columns of $W\in\R^{d\times r_{\rm AS}}$ are oblique to the coordinate axes.
Such a function can have nonzero ANOVA components at many orders, although its intrinsic variable is $r_{\rm AS}$-dimensional.
The active-subspace augmentation addresses this case by estimating the dominant gradient eigenspace and adding random features whose frequencies lie in that eigenspace.

After coordinate enrichment, let $K^\star\le K_{\max}$ be the terminal order, $\mathcal A^\star:=\mathcal A_{K^\star}$ the terminal coordinate family, and $\widetilde u:=\widehat u_{K^\star}$ the pre-augmentation predictor.
Draw $\{x_n^{\rm AS}\}_{n=1}^{N_{\rm AS}}\subset\Omega$ from $\mu$, independently of the collocation and screening samples, and set $\bm p_n:=\nabla\widetilde u(x_n^{\rm AS})\in\R^d$.
The empirical predictor covariance $\widehat{\mathbf C}\in\R^{d\times d}$ is
\[
\widehat{\mathbf C}
=
\frac{1}{N_{\rm AS}}\sum_{n=1}^{N_{\rm AS}} \bm p_n\bm p_n^\top .
\]
The eigenvalues are ordered as
$\widehat\nu_1\ge\widehat\nu_2\ge\cdots\ge\widehat\nu_d\ge0$.
Given an energy threshold $\tau_{\rm AS}\in(0,1)$ and an eigengap-ratio threshold $\gamma_{\rm gap}>1$, we search over $1\le r_{\rm AS}<d$ and select the smallest rank satisfying
\[
\sum_{i=1}^{r_{\rm AS}}\widehat\nu_i
\ge
\tau_{\rm AS}\sum_{i=1}^d\widehat\nu_i,
\qquad
\frac{\widehat\nu_{r_{\rm AS}}}{\widehat\nu_{r_{\rm AS}+1}}
\ge
\gamma_{\rm gap}.
\]
The ratio is $+\infty$ when $\widehat\nu_{r_{\rm AS}+1}=0<\widehat\nu_{r_{\rm AS}}$, while a zero-over-zero pair is inadmissible.
This ratio criterion is the computable counterpart of the additive gap used in the perturbation analysis.
When the test selects a rank $r_{\rm AS}$, let $\widehat W_{r_{\rm AS}}\in\R^{d\times r_{\rm AS}}$ contain the leading $r_{\rm AS}$ orthonormal eigenvectors.
We then append a block of $M_{\rm glob}$ global features,
\[
\phi_{{\rm glob},m}(x)
=
\sigma((\widehat W_{r_{\rm AS}}\xi_m)\cdot x+b_{{\rm glob},m}),
\qquad
m=1,\dots,M_{\rm glob},
\]
where $\xi_m\in\R^{r_{\rm AS}}$ is sampled in the estimated coordinate system and $b_{{\rm glob},m}\in\R$ is its bias.
The experiments use the coordinate-block sampling law with $r_{\rm AS}$ in place of $k$.
The coefficients of the subset-supported and global features are then refitted in the same joint least-squares problem~\eqref{eq:joint-ls}.

Thus predictor gradients alter the geometry of the trial space without introducing a nonlinear coefficient optimization.

\subsection{HA-RFM algorithm}

Algorithm~\ref{alg:harfm} assembles the coordinate enrichment, joint solves, stopping rules, and active-subspace test; Section~\ref{sec:experiments} gives the numerical settings and widths.

\begin{algorithm}[H]
\caption{Hierarchical ANOVA-RFM}
\label{alg:harfm}
\small
\begin{algorithmic}[1]
\Statex \textbf{Input:} PDE data $(L,f,g)$ and maximum order $1\le K_{\max}\le d$.
\Statex \textbf{Input:} Feature widths $M_1,\ldots,M_{K_{\max}},M_{\rm glob}$ and sampling laws.
\Statex \textbf{Input:} Sample sizes $N_{\rm int},N_{\rm bc},N_{\rm AS},M_{\rm QMC}$.
\Statex \textbf{Input:} Parameters $\omega_{\rm bc},\lambda_{\rm reg},\vartheta_{\rm Sob},\tau_{\rm AS},\gamma_{\rm gap}$; optional candidate cap; held-out functionals $\mathfrak E_{\rm int},\mathfrak E_{\rm bc}$; tolerances $\varepsilon_{\rm int}^{\rm stop},\varepsilon_{\rm bc}^{\rm stop}$.
\Statex \textbf{Output:} Approximation $\widehat u_{\rm HA}$ and terminal coordinate family $\mathcal A^\star$.
\State $\mathcal A\gets\{\emptyset\}\cup\{\{i\}:1\le i\le d\}$; instantiate the constant and singleton features.
\For{$K=1,\dots,K_{\max}$}
  \State Fit $\widehat u_K$ on the current $\mathcal A$ by the joint least-squares problem~\eqref{eq:joint-ls}; set $\widehat u_{\rm HA}\gets\widehat u_K$.
  \State \textbf{if} $\mathfrak E_{\rm int}(\widehat u_K)\le\varepsilon_{\rm int}^{\rm stop}$ and $\mathfrak E_{\rm bc}(\widehat u_K)\le\varepsilon_{\rm bc}^{\rm stop}$, or if $K=K_{\max}$, \textbf{break}.
  \State Estimate $\widehat T_S^{\rm closed}(\mathfrak r_K)$ for $|S|=K+1$, select $\mathcal S_{K+1}$ by~\eqref{eq:active-set}, and impose a prescribed candidate cap when used.
  \State \textbf{if} $\mathcal S_{K+1}=\emptyset$, \textbf{break}; otherwise set $\mathcal A\gets\mathcal A\cup\mathcal S_{K+1}$ and instantiate new features.
\EndFor
\State Set $K^\star\gets K$, $\mathcal A^\star\gets\mathcal A$, and $\widetilde u\gets\widehat u_{K^\star}$.
\State Form $\widehat{\mathbf C}$ from $\nabla\widetilde u$ and apply the energy/eigengap test.
\State \textbf{if} the test selects a rank $r_{\rm AS}<d$, append the global active-subspace features, re-solve~\eqref{eq:joint-ls}, and update $\widehat u_{\rm HA}$.
\State \Return $\widehat u_{\rm HA}$ and $\mathcal A^\star$.
\end{algorithmic}
\end{algorithm}

Steps~1--7 construct the terminal coordinate state, Step~8 applies the spectral test, and Step~9 appends the oblique block when indicated.
Section~\ref{sec:analysis} follows this order from selected-space error and screening recovery to covariance perturbation and the augmented solve.

\subsection{Picard extension}
\label{subsec:method-picard}

The Picard extension is the fixed-point wrapper that applies HA-RFM after each successive linearization.
The formulas below use first-order Taylor, hence Newton-type, linearizations of the nonlinear terms; the name Picard refers to the resulting sequence of linear elliptic solves.
The analysis requires only that the exact update map induced by the chosen linearization be contractive on the stated invariant set.
At iteration $\ell\ge0$, let $L_\ell$ and $f_\ell$ denote the operator and source obtained by linearizing at $u^\ell$; the next iterate solves
\[
L_\ell u^{\ell+1}=f_\ell,
\]
with the prescribed boundary data.
Each inner problem remains a linear coefficient fit and is assumed to satisfy the elliptic and sampled-LS stability inputs of Section~\ref{sec:analysis}.
For a differentiable reaction nonlinearity $\mathcal N:\R\to\R$, the affine approximation is
\[
\mathcal N(u^{\ell+1})
\approx
\mathcal N(u^\ell)+\mathcal N'(u^\ell)(u^{\ell+1}-u^\ell).
\]
For a differentiable gradient nonlinearity $\mathcal H:\R^d\to\R$, let $p\in\R^d$ denote its argument and $\nabla_p\mathcal H$ its gradient; the expansion is
\[
\mathcal H(\nabla u^{\ell+1})
\approx
\mathcal H(\nabla u^\ell)
+\,
\nabla_p \mathcal H(\nabla u^\ell)\cdot\nabla(u^{\ell+1}-u^\ell).
\]
The experiments apply these formulas to the linear-quadratic regulator Hamilton--Jacobi--Bellman equation and Allen--Cahn problem.
Theorem~\ref{thm:picard-inexact} propagates the inner HA-RFM error for contractive Picard maps, completing the link from the algorithm to the analysis.

\section{Approximation analysis}
\label{sec:analysis}

The analysis follows Algorithm~\ref{alg:harfm}: selected-space inputs yield $L^2$ and gradient error estimates; these lead to width, Sobol-screening recovery, active-subspace recovery, the augmented-solve estimate, and Picard propagation.
The estimates separate structural approximation, random-feature error, sampled least-squares stability, and eigengap inputs; complete proofs appear in the supplement.

\subsection{Analytic setting}
\label{subsec:analysis-inputs}

Fix a family $\mathcal A^\star$ returned by Algorithm~\ref{alg:harfm} and set $K^\star:=\max_{S\in\mathcal A^\star}|S|\le K_{\max}$; the screening analysis below gives conditions for recovering this family.
The empty component is represented by a constant, so $\emptyset\in\mathcal A^\star$, $n_0^{\rm blk}=M_0=1$.
For $k\ge2$, set $n_k^{\rm blk}:=|\{S\in\mathcal A^\star:\ |S|=k\}|$, and let
\[
\mathcal V_{\mathcal A^\star}
=
\sum_{S\in\mathcal A^\star}\mathcal V_{S,M_{|S|}}
\]
be the selected random-feature space.
All singleton blocks are retained, so $n_1^{\rm blk}=d$; the constant is absorbed into asymptotic width estimates.
We write
\[
N_{K^\star}:=\sum_{S\in\mathcal A^\star}M_{|S|}
\]
for the nominal selected width, namely the number of fitted coefficients before any rank deficiency is removed.
For nonhomogeneous Dirichlet data, the estimates apply after a fixed boundary lift; $L$ denotes the current linear or successive-linearized elliptic operator.
The analysis applies to uniformly elliptic strong-form operators
\[
Lv=-\sum_{i,j=1}^d a_{ij}(x)\partial_{ij}v+b(x)\cdot\nabla v+\kappa(x)v,
\]
where $\mathbf a(x):=(a_{ij}(x))\in\mathbb R^{d\times d}$ is the uniformly elliptic principal coefficient matrix, $b(x)\in\mathbb R^d$ is the drift, and $\kappa(x)\in\mathbb R$ is the reaction coefficient; all coefficients are regular enough for the stated residual and boundary-stability estimates.
Let $\operatorname{Tr}$ be the trace operator and $\mathcal Y_{\partial}$ a continuous trace space on $\partial\Omega$, for example $H^{1/2}(\partial\Omega)$.
The passage from the empirical boundary seminorm to $\|\operatorname{Tr}\cdot\|_{\mathcal Y_{\partial}}$ is a selected-space spectral-equivalence input.
We use the residual/trace quantity
\begin{equation}
\label{eq:combined-residual-norm}
\|v\|_{\mathcal R}
:=
\|Lv\|_{H^{-1}(\Omega)}
+\omega_{\rm th}\|\operatorname{Tr}v\|_{\mathcal Y_{\partial}} .
\end{equation}
Here $\omega_{\rm th}>0$ is the fixed trace scaling used in the analysis; the boundary penalty used in the sampled least-squares problem is accounted for through the empirical spectral-equivalence input below.

The estimate uses three inputs: selected-family approximation, residual/trace and block stability, and sampled least-squares (LS) stability.
A bare $\eta\in(0,1)$ is a generic failure budget; the subscripts in $\eta_{\rm feat}$, $\eta_{\rm LS}$, $\eta_{\rm scr}$, $\eta_{\rm AS,feat}$, and $\eta_{\rm AS}$ identify the corresponding feature, LS, screening, active-feature, and subspace events, and each combined budget is assumed to be less than one.
Descriptively subscripted $C$'s are positive constants whose stated dimension dependence is retained; their detailed operator, Gram, coherence, and sample-allocation dependencies are recorded in the supplement.

\begin{assumption}[selected ANOVA approximation]
\label{ass:selected-anova}
Let $u=\sum_{S\subseteq\mathcal I_d}u_S$ be the integral ANOVA decomposition with respect to the product reference measure $\mu$.
For a selected family $\mathcal A^\star$ of maximum order $K^\star$, set $u_{\mathcal A^\star}=\sum_{S\in\mathcal A^\star}u_S$ and let $\tau_{\rm sol},\tau_{\rm res}\ge0$ be the solution-tail and residual-tail tolerances.
Assume
\begin{equation}
\label{eq:assn-supdim}
\|u-u_{\mathcal A^\star}\|_{L^2(\Omega)}
\le
\tau_{\rm sol}\|u\|_{L^2(\Omega)},\qquad
\|u-u_{\mathcal A^\star}\|_{\mathcal R}
\le
\tau_{\rm res}\|u\|_{L^2(\Omega)} .
\end{equation}
For each $S\in\mathcal A^\star$ with $|S|=k$, let $e^{\rm rf}_{S,M_k}(R_{\rm rf},\eta)\ge0$ denote an $H^2(\Omega_S)$ approximation-error bound for the sampled feature span on an event with failure probability $\eta$, where $R_{\rm rf}>0$ is the random-feature frequency radius, and assume
\begin{equation}
\label{eq:assn-rf}
\inf_{v_S\in\mathcal V_{S,M_k}}\|u_S-v_S\|_{H^2(\Omega_S)}
\le
e^{\rm rf}_{S,M_k}(R_{\rm rf},\eta).
\end{equation}
Let $\bm c^\dagger$ be the concatenated comparison coefficient vector, satisfying
\[
\|\bm c^\dagger\|_2
\le
B_{\rm coef}(M_1,\dots,M_{K^\star},d,\eta).
\]
\end{assumption}

For a spectral or Barron-type component class matched to the feature distribution, we use
\[
e^{\rm rf}_{S,M_k}(R_{\rm rf},\eta)
\le
a_S\{E_{\rm trunc}(R_{\rm rf};u_S)+M_k^{-s_{\rm rf}/(2k)}\ell_k(\eta)\},
\]
where $a_S>0$ is the component amplitude, $E_{\rm trunc}$ the finite-window bias, and $\ell_k(\eta)>0$ a logarithmic factor.

\begin{assumption}[elliptic and block stability]
\label{ass:elliptic-block}
The current operator satisfies the boundary-stable estimate
\begin{equation}
\label{eq:assn-coercive}
\|e\|_{L^2(\Omega)}
\le
C_{\rm ell}^{\rm bd}\|e\|_{\mathcal R}
\qquad
\forall e\in H^1(\Omega),\ \operatorname{Tr}e\in\mathcal Y_{\partial}.
\end{equation}
Moreover, selected component errors $z_S\in H^2(\Omega_S)$, identified with the lifts $x\mapsto z_S(x_S)$ on $\Omega$, obey
\begin{equation}
\label{eq:assn-block-stability}
\left\|\sum_{S\in\mathcal A^\star}z_S\right\|_{\mathcal R}
\le
C_{\rm op}\Gamma_{K^\star}^{1/2}
\left(\sum_{S\in\mathcal A^\star}\|z_S\|_{H^2(\Omega_S)}^2\right)^{1/2}.
\end{equation}
\end{assumption}

The factor $\Gamma_{K^\star}$ measures block nonorthogonality in the residual/trace norm; the supplement gives a sampled coherence indicator and its limitations.

\begin{assumption}[sampled LS stability]
\label{ass:ls-stability}
Let $\widehat u$ be the Tikhonov-regularized collocation least-squares solution in $\mathcal V_{\mathcal A^\star}$ and let $u^\dagger$ be a comparison function in the same space.
Assume that, with probability at least $1-\eta_{\rm LS}$ over the interior and boundary samples,
\begin{equation}
\label{eq:assn-ls-stability}
\begin{aligned}
\|\widehat u-u^\dagger\|_{\mathcal R}
&\le
C_{\rm ls}
\left[
\|u^\dagger-u\|_{\mathcal R}
+C_{\rm samp}\chi_N
+C_{\rm reg}\lambda_{\rm reg}^{1/2}B_{\rm coef}
\right],
\\
	\chi_N
	&:=\chi_N(\eta_{\rm LS})
	=
\left(N_{\rm int}^{-1/2}+N_{\rm bc}^{-1/2}\right)
	\sqrt{\log(1/\eta_{\rm LS})} .
\end{aligned}
\end{equation}
\end{assumption}
The subscript in $\chi_N$ records dependence on $(N_{\rm int},N_{\rm bc})$; the displayed factor isolates the nominal sampling rate, while width, row-bound, and coherence dependence remains in $C_{\rm samp}$.
The supplement derives this input from population--empirical Gram spectral equivalence and gives a sufficient matrix concentration condition.
Any dimension dependence of the LS constants or $B_{\rm coef}$ is inherited by the width bounds.

\subsection{Selected-space error and width}
\label{subsec:approximation}
\label{subsec:complexity}

For every order with $n_k^{\rm blk}>0$, define the aggregate random-feature error
\[
\mathcal E_k(R_{\rm rf},M_k,\eta)^2
:=
\sum_{\substack{S\in\mathcal A^\star\\ |S|=k}}
e^{\rm rf}_{S,M_k}
\left(R_{\rm rf},\frac{\eta}{K^\star n_k^{\rm blk}}\right)^2,
\]
and set $\mathcal E_k=0$ when $n_k^{\rm blk}=0$.
This orderwise allocation makes all per-subset events simultaneous with total failure budget at most $\eta$; Theorem~\ref{thm:approximation} uses $\eta=\eta_{\rm feat}$.
The next result composes the three inputs above into the error estimate for the terminal coordinate predictor.

\begin{theorem}[selected-space error under boundary-stable LS]
\label{thm:approximation}
Suppose Assumptions~\ref{ass:selected-anova}--\ref{ass:ls-stability} hold for the order-$K^\star$ selected space.
Then, with probability at least $1-\eta_{\rm feat}-\eta_{\rm LS}$ over the random features and collocation samples,
\begin{align}
\label{eq:approx-bound}
\|u-\widehat u_{K^\star}\|_{L^2(\Omega)}
&\le
\left(\tau_{\rm sol}
+C_{\rm ell}^{\rm bd}C_{\rm ls}\tau_{\rm res}\right)
\|u\|_{L^2(\Omega)}
\nonumber\\
&\quad
+C_{\rm ell}^{\rm bd}(1+C_{\rm ls})
\Bigl[
C_{\rm op}\Gamma_{K^\star}^{1/2}
\left(
\sum_{k=1}^{K^\star}
\mathcal E_k(R_{\rm rf},M_k,\eta_{\rm feat})^2
\right)^{1/2}
\nonumber\\
&\qquad\qquad
+C_{\rm samp}\chi_N
+C_{\rm reg}\lambda_{\rm reg}^{1/2}B_{\rm coef}
\Bigr].
\end{align}
\end{theorem}

\begin{proof}[Proof sketch]
Represent $u_\emptyset$ exactly by the constant feature.
For nonempty $S$, choose $v_S$ from Assumption~\ref{ass:selected-anova} and set $v^\dagger=\sum_Sv_S$; the orderwise allocation gives simultaneous validity.
The selected ANOVA tail gives the two $\tau$ terms; the block estimate~\eqref{eq:assn-block-stability} converts componentwise $H^2$ errors to a residual/trace error for $u_{\mathcal A^\star}-v^\dagger$.
Assumption~\ref{ass:ls-stability} transfers the comparison error to the sampled LS solution, and the boundary-stable estimate~\eqref{eq:assn-coercive} converts the resulting residual/trace bound to $L^2$.
\end{proof}

The bound separates ANOVA truncation from feature, sampling, and regularization errors.
For the active-subspace step, set $\widetilde u:=\widehat u_{K^\star}$ and introduce the strengthened residual norm
\[
\|v\|_{\mathcal R_1}
:=
\|Lv\|_{L^2(\Omega)}
+\omega_{\rm th}\|\operatorname{Tr}v\|_{\mathcal Y_{\partial}}.
\]
Let $\mathcal X_1$ be a regularity space, with norm $\|\cdot\|_{\mathcal X_1}$, controlling the corresponding residual tail.
When the per-subset approximation input is available in this scale, denote its errors by $e_{S,M_k}^{\rm rf,(1)}$ and, for $n_k^{\rm blk}>0$, define
\[
\mathcal E_k^{(1)}(R_{\rm rf},M_k,\eta)^2
:=
\sum_{\substack{S\in\mathcal A^\star\\ |S|=k}}
e_{S,M_k}^{\rm rf,(1)}
\left(R_{\rm rf},\frac{\eta}{K^\star n_k^{\rm blk}}\right)^2,
\]
with the value zero when $n_k^{\rm blk}=0$.

\begin{proposition}[gradient control for the selected predictor]
\label{prop:predictor-gradient-control}
Assume the boundary-stability and sampled-LS estimates hold in $\mathcal R_1$, with
$\|\nabla e\|_{L^2(\Omega)}\le C_{\rm ell}^{(1)}\|e\|_{\mathcal R_1}$ for a constant $C_{\rm ell}^{(1)}>0$.
Let $C_\nabla>0$ absorb $C_{\rm ell}^{(1)}$ and the strengthened LS constants $C_{\rm samp}^{(1)}$ and $C_{\rm reg}^{(1)}$.
Assume also that
\[
\|u-u_{\mathcal A^\star}\|_{\mathcal R_1}\le \tau_1\|u\|_{\mathcal X_1},
\]
and that the per-subset estimates defining $\mathcal E_k^{(1)}$ hold.
Then, on the corresponding feature and sampled-LS stability events,
\[
\begin{aligned}
\|\nabla(u-\widetilde u)\|_{L^2(\Omega)}
&\le
C_\nabla\Bigl[
\tau_1\|u\|_{\mathcal X_1}
+\Gamma_{K^\star}^{1/2}
	\left(
	\sum_{k=1}^{K^\star}
	\mathcal E_k^{(1)}(R_{\rm rf},M_k,\eta_{\rm feat})^2
	\right)^{1/2}
\\
&\qquad
+C_{\rm samp}^{(1)}\chi_N
+C_{\rm reg}^{(1)}\lambda_{\rm reg}^{1/2}B_{\rm coef}
\Bigr].
\end{aligned}
\]
If the events use parameters $\eta_{\rm feat}$ and $\eta_{\rm LS}$, the bound holds with probability at least $1-\eta_{\rm feat}-\eta_{\rm LS}$.
\end{proposition}

Balancing the $L^2$ contributions against an effective tolerance yields the fitted width.
Normalize $\|u\|_{L^2(\Omega)}=1$, let $0<\varepsilon_{\rm tar}\le1$ be the prescribed accuracy, and set
\[
A_k^2:=\sum_{\substack{S\in\mathcal A^\star\\ |S|=k}}a_S^2,
\qquad
\varepsilon_{\rm eff}
:=
\varepsilon_{\rm tar}-\tau_{\rm sol}-C_{\rm ell}^{\rm bd}C_{\rm ls}\tau_{\rm res},
\qquad
\eta_{\rm app}:=\eta_{\rm feat}+\eta_{\rm LS}.
\]
Here and below, $\operatorname{polylog}$ denotes polynomial factors in the logarithms of its displayed arguments.

\begin{theorem}[trial-space width]
\label{thm:complexity}
Assume the hypotheses of Theorem~\ref{thm:approximation}, the component-class rate with random-feature smoothness $s_{\rm rf}>0$, and bounded $\Gamma_{K^\star}$, LS constants, Gram nondegeneracy/coherence, and $B_{\rm coef}$.
Suppose $\varepsilon_{\rm eff}>0$ and that the aggregate frequency-truncation, sampling, and regularization terms are each bounded by prescribed fixed fractions of $\varepsilon_{\rm eff}$.
Then there exist integer widths satisfying
\[
M_k\asymp
\max\left\{1,
\left(\frac{\Gamma_{K^\star}^{1/2}A_k}{\varepsilon_{\rm eff}}\right)^{2k/s_{\rm rf}}
\right\},
\]
up to logarithmic factors, for which $\|u-\widehat u_{K^\star}\|_{L^2(\Omega)}\le\varepsilon_{\rm tar}$ with probability at least $1-\eta_{\rm app}$.
Consequently,
\begin{align}
\label{eq:complexity-general}
N_{K^\star}
&\lesssim
d
\max\left\{1,\left(
\frac{\Gamma_{K^\star}^{1/2}A_1}{\varepsilon_{\rm eff}}
\right)^{2/s_{\rm rf}}\right\}
\operatorname{polylog}(1+d,1/\varepsilon_{\rm tar},1/\eta_{\rm feat},1/\eta_{\rm LS})
\nonumber\\
&\quad
+\sum_{k=2}^{K^\star}
n_k^{\rm blk}
\max\left\{1,\left(
\frac{\Gamma_{K^\star}^{1/2}A_k}{\varepsilon_{\rm eff}}
\right)^{2k/s_{\rm rf}}\right\}
\operatorname{polylog}(1+n_k^{\rm blk},1/\varepsilon_{\rm tar},1/\eta_{\rm feat},1/\eta_{\rm LS}).
\end{align}
\end{theorem}

\begin{proof}[Proof sketch]
Balance the approximation term in Theorem~\ref{thm:approximation} to the target tolerance at each order and impose the integer floor $M_k\ge1$.
Adding the constant feature and summing $n_k^{\rm blk}M_k$ over selected nonempty blocks gives~\eqref{eq:complexity-general}.
The bound $n_k^{\rm blk}\le O(d^k)$ gives the dense width; bounded aggregate amplitudes remove the extra dimension factor in screened or energy-concentrated regimes.
\end{proof}

For the uniform dense family with fixed $K^\star\ge1$, $n_k^{\rm blk}\le\binom dk$ and $a_S\le a_{\max}$ reduce the theorem to
\[
N_{K^\star}
=
O\!\left(
d^{K^\star(1+K^\star/s_{\rm rf})}
\varepsilon_{\rm eff}^{-2K^\star/s_{\rm rf}}
\operatorname{polylog}(1+d,1/\varepsilon_{\rm tar},1/\eta_{\rm feat},1/\eta_{\rm LS})
\right).
\]
Now suppose that $K^\star$ is fixed.
If $n_k^{\rm blk}$, $A_k$, and the stability constants are dimension-independent for $k\ge2$, then so is the higher-order augmentation, apart from the singleton contribution $dM_1$.

The theorem controls fitted width, not screening work: a full order-$k$ screen uses $O(M_{\rm QMC}\binom{d}{k})$ residual evaluations, up to shared base evaluations.
Candidate restrictions or caps reduce this cost but require separate recovery assumptions.

\subsection{Residual Sobol screening}
\label{subsec:identification}

Residual screening selects the higher-order family required by Theorem~\ref{thm:complexity} when an ideal level-$k$ correction signal has a sufficient closed-index gap and the fitted residual is a small perturbation of that signal.

Fix an order $k$, set $\mathcal C_k:=\{S\subseteq\mathcal I_d:\ |S|=k\}$, and, for $q\in L^2_\mu(\Omega)$ with positive variance, write
$T_S:=T_S^{\rm closed}(q)$ and $T_{\max}:=\max_{S\in\mathcal C_k}T_S$.
Let $\widehat T_S$ be the empirical estimate.
A target family $\mathcal A_{k,{\rm closed}}^{\rm gap}\subseteq\mathcal C_k$ is separated at threshold $\vartheta_{\rm Sob}\in(0,1)$ with margin $m_{\rm Sob}>0$ if
\[
\begin{aligned}
T_S-\vartheta_{\rm Sob}T_{\max}&\ge m_{\rm Sob}
&& (S\in\mathcal A_{k,{\rm closed}}^{\rm gap}),\\
T_S-\vartheta_{\rm Sob}T_{\max}&\le-m_{\rm Sob}
&& (S\notin\mathcal A_{k,{\rm closed}}^{\rm gap}).
\end{aligned}
\]

\begin{proposition}[relative closed-Sobol screening]
\label{prop:identification}
If the target family satisfies the displayed separation and
$\max_{S\in\mathcal C_k}|\widehat T_S-T_S|\le e_k<m_{\rm Sob}/(1+\vartheta_{\rm Sob})$, then
\[
\widehat{\mathcal A}_k=\{S\in\mathcal C_k:\widehat T_S>\vartheta_{\rm Sob}\max_{S'\in\mathcal C_k}\widehat T_{S'}\},
\]
identifies $\mathcal A_{k,{\rm closed}}^{\rm gap}$ exactly.
For bounded Saltelli estimators based on $M_{\rm pf}$ independent pick-freeze rows, this event follows under the scaling
\[
M_{\rm pf}\gtrsim
\frac{G_\infty^4(1+\vartheta_{\rm Sob})^2}{v_{\min}^2m_{\rm Sob}^2}
\log\frac{2|\mathcal C_k|}{\eta},
\qquad
\|q\|_\infty\le G_\infty,\quad \operatorname{Var}(q)\ge v_{\min}>0 .
\]
\end{proposition}

The computations instead use $M_{\rm QMC}$ randomized QMC points to estimate the same indices; clipping the nonnegative estimates preserves the uniform error event.

\begin{proof}[Proof sketch]
On the uniform concentration event, active indices are selected, while inactive
indices lie below the threshold by the displayed margin.
The sampling condition is the independent pick-freeze event union-bounded over $\mathcal C_k$.
\end{proof}

To connect this ideal criterion to residual screening, for each $k=2,\ldots,K^\star$ let $q_k^{\rm ideal}$ be the level-$k$ correction signal and $\mathcal A_{k,{\rm closed}}^{\rm gap}$ its nonempty separated family.
Set $n_{k,{\rm gap}}^{\rm blk}:=|\mathcal A_{k,{\rm closed}}^{\rm gap}|$ and
\[
\mathcal A_{\rm ideal}^\star
:=\{\emptyset\}\cup\{\{i\}:i\in\mathcal I_d\}
\cup\bigcup_{k=2}^{K^\star}\mathcal A_{k,{\rm closed}}^{\rm gap}.
\]
Let $\mathfrak L_{k,{\rm gap}}$ be the polylog factor in~\eqref{eq:complexity-general} evaluated at $n_{k,{\rm gap}}^{\rm blk}$.
The nonempty-level condition ensures that Algorithm~\ref{alg:harfm} reaches every order through $K^\star$.
For the fitted residual $\mathfrak r_{k-1}:=f-L\widehat u_{k-1}$, suppose uniformly over $k=2,\ldots,K^\star$ that
\[
\|\mathfrak r_{k-1}-q_k^{\rm ideal}\|_{L^2_\mu(\Omega)}
\le
\min\left\{\frac{v_{\min}}{16G},
\frac{v_{\min}m_{\rm Sob}}{2C_{\rm sig}G(1+\vartheta_{\rm Sob})}\right\},
\]
where $\|\mathfrak r_{k-1}\|_{L^2_\mu(\Omega)},\|q_k^{\rm ideal}\|_{L^2_\mu(\Omega)}\le G$, $\operatorname{Var}(q_k^{\rm ideal})\ge v_{\min}>0$, and $C_{\rm sig}>0$ is the absolute constant in the closed-Sobol perturbation bound.
\begin{corollary}[screened width under residual perturbation]
\label{cor:screening-complexity}
Assume the inputs of Theorems~\ref{thm:approximation}--\ref{thm:complexity} with $\mathcal A^\star=\mathcal A_{\rm ideal}^\star$.
The levelwise closed-index families coincide with the intended order-$k$ ANOVA families when the required lower-order removal conditions hold.
If the preceding signal and residual conditions hold at every screened order and the independent pick--freeze concentration events have margin $m_{\rm Sob}/2$ and total failure budget $\eta_{\rm scr}$, then screening recovers this family and, with probability at least $1-\eta_{\rm scr}-\eta_{\rm feat}-\eta_{\rm LS}$,
\[
N_{K^\star}
\lesssim
dM_1+
\sum_{k=2}^{K^\star}
n_{k,{\rm gap}}^{\rm blk}
\max\left\{1,\left(
\frac{\Gamma_{K^\star}^{1/2}A_k}{\varepsilon_{\rm eff}}
\right)^{2k/s_{\rm rf}}\right\}
\mathfrak L_{k,{\rm gap}}.
\]
\end{corollary}

Hence the screening step replaces the full order-$k$ family by the recovered family of size $n_{k,{\rm gap}}^{\rm blk}$ without changing the selected-space accuracy conclusion.

\subsection{Active-subspace bounds}
\label{subsec:as-analysis}

For active-subspace identification, write $\mathbf C_u:=\mathbf C[u]=\E[\nabla u(\nabla u)^\top]$ and $\widetilde{\mathbf C}:=\mathbf C[\widetilde u]$, with expectation taken with respect to the product reference measure, and abbreviate $\|\cdot\|_{L^2_\mu(\Omega)}$ by $\|\cdot\|_{L^2_\mu}$ in this subsection.
If $\widetilde u$ is the selected predictor, then
\begin{equation}
\label{eq:predictor-covariance}
\|\mathbf C_u-\widetilde{\mathbf C}\|_{\rm op}
\le
\left(\|\nabla u\|_{L^2_\mu}+\|\nabla\widetilde u\|_{L^2_\mu}\right)
\|\nabla(u-\widetilde u)\|_{L^2_\mu}.
\end{equation}
On the analytic domain $(0,1)^d$, the reference measure is normalized Lebesgue measure, so $L^2(\Omega)=L^2_\mu(\Omega)$; hence Proposition~\ref{prop:predictor-gradient-control} supplies the predictor term in~\eqref{eq:predictor-covariance}.

Assume $\mathbf C_u$ has eigengap $\Delta:=\nu_{r_{\rm AS}}-\nu_{r_{\rm AS}+1}>0$.
Let $\widehat{\mathbf C}$ be the empirical covariance formed from $\widetilde u$ using samples independent of the least-squares and screening samples, conditional on the fitted predictor, and suppose
\[
\|\widehat{\mathbf C}-\widetilde{\mathbf C}\|_{\rm op}
+\|\widetilde{\mathbf C}-\mathbf C_u\|_{\rm op}\le \Delta/2 .
\]

\begin{corollary}[subspace identification with predictor error]
\label{cor:as-predictor}
Under the preceding eigengap and covariance-perturbation conditions, for $W$ and $\widehat W_{r_{\rm AS}}$ spanning the exact and empirical leading eigenspaces,
\[
\|\sin\Theta(\widehat W_{r_{\rm AS}},W)\|_{\rm op}
\le
\frac{2}{\Delta}
\left[
\|\widehat{\mathbf C}-\widetilde{\mathbf C}\|_{\rm op}
+\left(\|\nabla u\|_{L^2_\mu}+\|\nabla\widetilde u\|_{L^2_\mu}\right)
\|\nabla(u-\widetilde u)\|_{L^2_\mu}
\right].
\]
If $\|\nabla\widetilde u(x)\|_2\le G_\nabla$ almost surely, assume also that $N_{\rm AS}\gtrsim\log(d/\eta)$.
The first term is then $O(G_\nabla^2\sqrt{\log(d/\eta)/N_{\rm AS}})$ with probability at least $1-\eta$.
\end{corollary}

Consequently, the two spectral tests in Section~\ref{subsec:method-active-subspace} are stable whenever their energy and eigengap-ratio inequalities hold with margins exceeding the induced eigenvalue perturbations.

The augmented estimate uses a fixed coordinate-plus-ridge comparison; its coordinate term is not the exact selected ANOVA projection $u_{\mathcal A^\star}$.
For second-order strong-form residuals on bounded domains, the ridge perturbation follows from smooth coefficients and bounded derivatives of $h$ through order three.

Let $W\in\R^{d\times r_{\rm AS}}$ be the exact leading basis from Corollary~\ref{cor:as-predictor}, and let the coordinate components $u_{S,{\rm coord}}$ depend only on $x_S$ for $S\in\mathcal A^\star$.
Set $u_{\rm coord}:=\sum_{S\in\mathcal A^\star}u_{S,{\rm coord}}$ and consider the fixed comparison decomposition
\[
u(x)=u_{\rm coord}(x)+h(W^\top x)+u_\perp(x),
\qquad
\|u_\perp\|_{L^2(\Omega)}+C_{\rm ell}^{\rm bd}\|u_\perp\|_{\mathcal R}
\le\tau_{\rm mix}\|u\|_{L^2(\Omega)}.
\]
Here $\tau_{\rm mix}\ge0$ is the unresolved mixed-structure tolerance; neither orthogonality nor uniqueness is required.
For each nonempty coordinate component, let $e_{S,M_k}^{\rm rf,c}(R_{\rm rf},\eta)$ satisfy~\eqref{eq:assn-rf} with $u_S$ replaced by $u_{S,{\rm coord}}$, and, for $n_k^{\rm blk}>0$, define
\[
\mathcal E_k^{\rm c}(R_{\rm rf},M_k,\eta)^2
:=\sum_{\substack{S\in\mathcal A^\star\\|S|=k}}
e_{S,M_k}^{\rm rf,c}
\left(R_{\rm rf},\frac{\eta}{K^\star n_k^{\rm blk}}\right)^2,
\]
with the value zero when $n_k^{\rm blk}=0$.

Let $\widehat W_{r_{\rm AS}}$ be the empirical active-subspace basis and set
\[
\theta_{\rm AS}:=\|\sin\Theta(\widehat W_{r_{\rm AS}},W)\|_{\rm op}.
\]
Choose an orthogonal alignment $Q\in\R^{r_{\rm AS}\times r_{\rm AS}}$ such that
$\|W-\widehat W_{r_{\rm AS}}Q^\top\|_{\rm op}\le\sqrt{2}\,\theta_{\rm AS}$, and set $h_Q(z):=h(Qz)$.
Assume
\[
\|h(W^\top x)-h_Q(\widehat W_{r_{\rm AS}}^\top x)\|_{\mathcal R}
\le C_{\rm ridge}L_hX_\Omega\theta_{\rm AS},
\]
where $L_h$ bounds the derivatives entering $\mathcal R$ and $X_\Omega=(\E\|X\|_2^2)^{1/2}$.
Let $\mathcal V_{\rm AS}$ be the span of $M_{\rm glob}$ global random features in the estimated coordinates and define
\[
\mathcal E_{\rm AS}
:=a_{\rm AS}\{E_{\rm trunc}^{\rm AS}(R_{\rm rf};h)
+M_{\rm glob}^{-s_{\rm rf}/(2r_{\rm AS})}\ell_{\rm AS}(\eta_{\rm AS,feat})\}.
\]
Here $a_{\rm AS}>0$ is the ridge-class amplitude, $E_{\rm trunc}^{\rm AS}$ is the active-coordinate frequency-window bias, and $\ell_{\rm AS}$ is the logarithmic feature-event factor.
Assume, with probability at least $1-\eta_{\rm AS,feat}$, that
\[
\inf_{v_{\rm AS}\in\mathcal V_{\rm AS}}
\|h_Q(\widehat W_{r_{\rm AS}}^\top x)-v_{\rm AS}\|_{\mathcal R}
\le \mathcal E_{\rm AS}.
\]
Take $v_\emptyset=u_{\emptyset,{\rm coord}}$ exactly and, on the coordinate-feature event, choose $v_S\in\mathcal V_{S,M_k}$ for nonempty $S$; set $z_{S,{\rm coord}}:=u_{S,{\rm coord}}-v_S$.
Assume these comparison errors satisfy the augmented block estimate with nonorthogonality factor $\Gamma_{\rm aug}\ge1$,
\[
\left\|\sum_{S\in\mathcal A^\star}z_{S,{\rm coord}}\right\|_{\mathcal R}
\le C_{\rm op}\Gamma_{\rm aug}^{1/2}
\left(\sum_{S\in\mathcal A^\star}\|z_{S,{\rm coord}}\|_{H^2(\Omega_S)}^2\right)^{1/2}.
\]
For $\mathcal V_{\mathcal A^\star}+\mathcal V_{\rm AS}$, let $B_{\rm coef}^{\rm aug}$ and $\chi_N^{\rm aug}$ denote the comparison-coefficient bound and sampled-LS fluctuation, respectively, and set $C_{\rm mix}:=\max\{1,C_{\rm ls}\}$.

\begin{theorem}[augmented active-subspace error]
\label{thm:as-augmented-error}
Under the preceding decomposition, rotation, feature-approximation, and augmented block-stability assumptions, suppose the joint solve on $\mathcal V_{\mathcal A^\star}+\mathcal V_{\rm AS}$ satisfies Assumption~\ref{ass:ls-stability}.
Let the subspace event in Corollary~\ref{cor:as-predictor} have failure budget $\eta=\eta_{\rm AS}$.
Then, on the feature, least-squares, and subspace-identification events, with probability at least $1-\eta_{\rm feat}-\eta_{\rm AS,feat}-\eta_{\rm LS}-\eta_{\rm AS}$,
\begin{align}
\label{eq:as-augmented-error}
\|u-\widehat u_{\rm HA}\|_{L^2(\Omega)}
&\le
C_{\rm mix}\tau_{\rm mix}\|u\|_{L^2(\Omega)}
\nonumber\\
&\quad
+C_{\rm ell}^{\rm bd}(1+C_{\rm ls})
\Bigl[
C_{\rm op}\Gamma_{\rm aug}^{1/2}
	\left(
	\sum_{k=1}^{K^\star}
	\mathcal E_k^{\rm c}(R_{\rm rf},M_k,\eta_{\rm feat})^2
	\right)^{1/2}
\nonumber\\
&\qquad
+C_{\rm ridge}L_hX_\Omega\theta_{\rm AS}
+\mathcal E_{\rm AS}
\nonumber\\
&\qquad
+C_{\rm samp}\chi_N^{\rm aug}
+C_{\rm reg}\lambda_{\rm reg}^{1/2}B_{\rm coef}^{\rm aug}
\Bigr],
\end{align}
where $\theta_{\rm AS}$ is bounded by Corollary~\ref{cor:as-predictor}.
\end{theorem}

If $\mathcal A^\star$ is produced by residual screening rather than fixed in advance, intersecting with the event in Corollary~\ref{cor:screening-complexity} adds $\eta_{\rm scr}$ to the displayed failure budget.

\begin{proof}[Proof sketch]
Add and subtract $h_Q(\widehat W_{r_{\rm AS}}^\top x)$ and combine its subspace-rotation and low-dimensional random-feature errors.
Control $u_{\rm coord}$ by the per-subset construction and block estimate, then apply the augmented LS stability and boundary-stable elliptic estimate.
\end{proof}

\subsection{Picard iteration with HA-RFM inner solves}
\label{subsec:picard-analysis}

For semilinear equations, the preceding linear-solve estimates become inner errors in the fixed-point wrapper defined by the chosen successive linearization.

Let $\mathcal X$ be a normed solution space, with norm $\|\cdot\|_{\mathcal X}$, controlling the nonlinear terms, including gradients when needed.
Assume the exact update map $\mathcal T:\mathcal X\to\mathcal X$ has a fixed point $u^\star$ and is a contraction with modulus $0<L_{\mathcal T}<1$ on a closed invariant ball $\mathbb B_{\rm Pic}\subset\mathcal X$ containing the iterates.
Let $u^{\ell+1}:=\widehat{\mathcal T}_\ell(u^\ell)$ satisfy
\[
\|\widehat{\mathcal T}_\ell(u^\ell)-\mathcal T(u^\ell)\|_{\mathcal X}
\le\delta_{{\rm Pic},\ell}.
\]

\begin{theorem}[Picard iteration with inexact HA-RFM solves]
\label{thm:picard-inexact}
Under the preceding contraction and inexact-solve conditions, for $n\ge1$ the iterates satisfy
\[
\|u^n-u^\star\|_{\mathcal X}
\le
L_{\mathcal T}^n\|u^0-u^\star\|_{\mathcal X}
+\sum_{\ell=0}^{n-1}L_{\mathcal T}^{n-1-\ell}\delta_{{\rm Pic},\ell} .
\]
If $\delta_{{\rm Pic},\ell}\le\delta_{\rm Pic}$ for $0\le\ell<n$, the accumulated inner-solve error is bounded by $(1-L_{\mathcal T}^n)\delta_{\rm Pic}/(1-L_{\mathcal T})$.
\end{theorem}

\begin{proof}[Proof sketch]
Subtract $u^\star=\mathcal T(u^\star)$ and use
\[
\|u^{\ell+1}-u^\star\|_{\mathcal X}
\le
\delta_{{\rm Pic},\ell}+L_{\mathcal T}\|u^\ell-u^\star\|_{\mathcal X}.
\]
Iterating the scalar recurrence gives the result.
\end{proof}

The same recurrence covers variable linearized operators with a uniform contraction constant and invariant ball.
For the fixed-point map $\mathcal T(v)=L_0^{-1}(f-\mathcal N(v))$, it suffices that $L_0^{-1}:\mathcal X^\ast\to\mathcal X$ have norm $C_{\rm inv}$, where $\mathcal X^\ast$ is the dual space, and that $\mathcal N$ be $L_{\mathcal N}$-Lipschitz on the ball with $C_{\rm inv}L_{\mathcal N}<1$.
The preceding linear-solve results supply $\delta_{{\rm Pic},\ell}$ when their bounds hold in $\mathcal X$ and their constants are uniform on $\mathbb B_{\rm Pic}$.

\section{Numerical experiments}
\label{sec:experiments}

The experiments first test selected-space error under ANOVA truncation and predictor-gradient identification of oblique coordinates, then assess equal-width RFM comparisons and the Picard extension, and finally examine stability, residual-Sobol recovery, and densely coupled regimes requiring broader coordinate spaces.

Unless stated otherwise, we use $\sigma=\tanh$, the unnormalized-system values $\lambda_{\rm impl}=10^{-8}$ and $\omega_{\rm bc,impl}=100$, frequencies sampled entry-wise from $[-R_{\rm rf},R_{\rm rf}]^k$, and biases sampled from $[-\pi,\pi]$; $R_{\rm rf}=\pi/2$ on $(-1,1)^d$ and $R_{\rm rf}=\pi$ on $(0,1)^d$.
All selected coefficients are fitted jointly by the least-squares formulation using QR with column pivoting.
Interior sample counts follow the selected width, with boundary counts ranging from about one fifth to two fifths of the interior counts; the largest runs use at most $1.8\times10^4$ interior and $3000$ boundary points.
Unless stated otherwise, covariance estimates use $N_{\rm AS}=4096$ samples.
Errors are relative $L^2$ errors on $20\,000$ independent Monte Carlo test points, except the closed-form ANOVA tests in Table~\ref{tab:exp-verification}, which use $10\,000$ points.
CPU times are wall-clock measurements on an Apple M3 Pro MacBook Pro with 12 cores and 18 GB unified memory.

Full-dimensional RFM is matched to the corresponding HA-RFM width.
The PINN baseline is a four-hidden-layer, width-128 tanh network trained by Adam for $20{,}000$ iterations with learning rate $10^{-3}$, $2048$ interior points per iteration, boundary weight $100$, and $n_{\partial}= \max\{8,\lfloor 64/d\rfloor\}$ points per boundary face.
We report ANOVA tails, alignment, covariance-signal comparisons, and eigengaps; source-covariance runs use problem-derived signals, whereas predictor-gradient runs use $\nabla\widetilde u$.
The studies prescribe coordinate order and, where indicated, a rank-one augmentation to isolate the two allocation mechanisms.
Under the threshold $\gamma_{\rm gap}=10$ in Algorithm~\ref{alg:harfm}, the measured eigengaps cleanly separate the oblique and dense regimes.

\subsection{ANOVA truncation and selected-space error}
\label{subsec:exp-verification}

Two closed-form solutions with known ANOVA decompositions probe the structural term in Theorem~\ref{thm:approximation}.
The first is a product target on $(0,1)^d$:
$-\Delta u + \kappa(x)u=f$, $\kappa(x)=\sum_i \cos(\pi x_i)$, homogeneous boundary data, and
$u_{\rm exact}(x)=\prod_i \sin(\pi x_i)$.
Its order-$k$ components have variance $\mathrm{Var}(v)^k(2/\pi)^{2(d-k)}$, where $v(t)=\sin(\pi t)-2/\pi$; together with their multiplicity $\binom{d}{k}$ and $\mathrm{Var}(v)/(2/\pi)^2\approx0.234$, this gives the truncation tail.

The second target has exact superposition dimension two:
\[
u_{\rm exact}(x)
= \sum_i \sin(\pi x_i)
+ \tfrac{1}{2}\sum_{i<j}\sin(\pi x_i)\sin(\pi x_j)
\]
on $(0, 1)^d$ with $\kappa \equiv 1$ and matching non-homogeneous Dirichlet datum; no order-$\ge 3$ ANOVA component exists by construction.
Table~\ref{tab:exp-verification} reflects Theorem~\ref{thm:approximation}.
For the product target, increasing $K$ from one to two reduces the error at $d=5$ but remains truncation-limited at $d=10$, where the unresolved higher-order tail is large.
For the exact order-two target, $K=2$ removes the structural tail, leaving feature, sampling, and regularization errors.
Here $M_{\rm top}$ denotes the number of random features assigned to each retained top-order subset, and $N_K$ denotes the total number of fitted coefficients in the reported space.

\begin{table}[!htbp]
\centering
\caption{ANOVA truncation and selected-space error in the setting of Theorem~\ref{thm:approximation}. Left: product target with nonzero ANOVA mass at every order. Right: exact order-two target with no higher-order ANOVA component.}
\label{tab:exp-verification}
\scriptsize
\setlength{\tabcolsep}{4pt}
\begin{tabular}[c]{ccrrrr}
\toprule
\multicolumn{6}{c}{\textbf{Product target}} \\
\cmidrule(lr){1-6}
\textbf{$d$} & \textbf{$K$} & \textbf{$M_{\rm top}$} & \textbf{$N_K$} & \textbf{rel.\ $L^2$} & \textbf{CPU (s)} \\
\midrule
5 & 1 & 200 & 1\,000 & $5.5{\times}10^{-1}$ & 1.3 \\
5 & 2 & 100 & 2\,000 & $2.7{\times}10^{-1}$ & 9.4 \\
\midrule
10 & 1 & 150 & 1\,500 & $7.9{\times}10^{-1}$ & 2.5 \\
10 & 2 & 60  & 4\,200 & $7.6{\times}10^{-1}$ & 65 \\
\bottomrule
\end{tabular}
\qquad
\begin{tabular}[c]{ccrrr}
\toprule
\multicolumn{5}{c}{\textbf{Exact order-two target ($K=2$)}} \\
\cmidrule(lr){1-5}
\textbf{$d$} & \textbf{$M_1,M_2$} & \textbf{$N_K$} & \textbf{rel.\ $L^2$} & \textbf{CPU (s)} \\
\midrule
5  & $100,80$ & 1\,300 & $7.1{\times}10^{-6}$ & 5.2 \\
10 & $40,40$  & 2\,200 & $1.6{\times}10^{-4}$ & 14 \\
10 & $80,80$  & 4\,400 & $6.2{\times}10^{-6}$ & 75 \\
10 & $100,80$ & 4\,600 & $6.2{\times}10^{-6}$ & 86 \\
\bottomrule
\end{tabular}
\end{table}

\subsection{Oblique low-rank structure}
\label{subsec:exp-oblique}

The next solution depends on a single oblique coordinate.
Consider the Poisson problem $-\Delta u=f$ on $(-1,1)^d$ with manufactured solution
\begin{equation}
\label{eq:oblique-poisson-u}
u_{\rm exact}(x) \;=\; s^2 + \sin s, \quad s := \frac{1}{d} \sum_{i=1}^{d} x_i,
\end{equation}
source $f(x)=(1/d)(\sin s-2)$, and Dirichlet datum $g\equiv u_{\rm exact}$.
The solution has effective dimension one, with active direction $(1,\dots,1)/\sqrt d$ oblique to the coordinate axes.
Its coordinate ANOVA expansion nevertheless has nonzero components at every order, directly testing the active-coordinate augmentation.

\paragraph{Covariance step}
Gradients of the initial ANOVA fit define the predictor covariance, and Table~\ref{tab:predictor-as} reports the absolute inner product between the estimated and true directions.
The label ``all'' retains all pairs and triples, ``residual'' uses the displayed Sobol-screened counts, and ``all pairs'' uses every pair but no triples.
The $d=30$ and $d=50$ cases use smaller per-pair widths to separate the four-decimal-place predictor alignment from higher-order screening; $N_K$ includes the active block.

\begin{table}[!htbp]
\centering
\caption{Predictor-gradient active-subspace alignment on the oblique Poisson problem. The direction is computed from $\nabla\widetilde u$; alignment exceeds $0.9999$ through $d=50$.}
\label{tab:predictor-as}
\scriptsize
\setlength{\tabcolsep}{3pt}
\begin{tabular}{ccccccc}
\toprule
\textbf{$d$} & \textbf{screening} & \textbf{selected pairs} & \textbf{selected triples} & \textbf{$N_K$} & \textbf{rel.\ $L^2$} & \textbf{alignment} \\
\midrule
10 & all      & 45  & 120 & 3\,780 & $4.84\times10^{-7}$ & $>0.9999$ \\
10 & residual & 3   & 120 & 2\,520 & $1.34\times10^{-6}$ & $>0.9999$ \\
15 & all      & 105 & 455 & 8\,865 & $8.20\times10^{-7}$ & $>0.9999$ \\
15 & residual & 99  & 400 & 8\,055 & $4.80\times10^{-6}$ & $>0.9999$ \\
20 & residual & 48  & 400 & 5\,990 & $4.10\times10^{-4}$ & $>0.9999$ \\
25 & residual & 300 & 300 & 7\,930 & $2.87\times10^{-5}$ & $>0.9999$ \\
30 & all pairs & 435 & 0 & 5\,430 & $3.49\times10^{-5}$ & $>0.9999$ \\
50 & all pairs & 1\,225 & 0 & 7\,405 & $4.54\times10^{-4}$ & $>0.9999$ \\
\bottomrule
\end{tabular}
\end{table}

\paragraph{Random oblique direction}
To remove the symmetry of~\eqref{eq:oblique-poisson-u}, we repeat the test with a dense random unit vector $w\in\mathbb S^{d-1}$ obtained from normalized i.i.d.\ $\mathcal N(0,1)$ entries and held fixed across methods; $\max_i|w_i|$ measures its coordinate concentration.
We set
\[
u_{\rm exact}(x)=(w^\top x)^2+\sin(w^\top x),\qquad -\Delta u=\sin(w^\top x)-2 .
\]
Using $\nabla\widetilde u$, Table~\ref{tab:random-oblique} gives alignment at least $0.999$ through $d=50$.
Adding only $40$ active-coordinate features to coordinate spaces of widths $4800$, $5550$, and $6400$ increases width by $0.83\%$, $0.72\%$, and $0.63\%$.
It reduces error by factors of $39$, $25$, and $14$ over the coordinate precursor and $100$, $65$, and $34$ over equal-width full RFM.

\begin{table}[!htbp]
\centering
\caption{Near-width ablation for a random oblique ridge direction. Covariance is computed from $\nabla\widetilde u$. Unknowns match predictor-AS HA-RFM and full RFM; the $K=2$ precursor has only $40$ fewer, so the active block adds less than $1\%$.}
\label{tab:random-oblique}
\scriptsize
\setlength{\tabcolsep}{4pt}
\begin{tabular}{ccccccc}
\toprule
\textbf{$d$} & \textbf{unknowns} & \textbf{$\max_i |w_i|$} & \textbf{HA-RFM $K=2$} & \textbf{HA-RFM + pred.\ AS} & \textbf{full RFM} & \textbf{alignment} \\
\midrule
20 & 4\,840 & 0.582 & $1.14{\times}10^{-1}$ & $2.95{\times}10^{-3}$ & $2.94{\times}10^{-1}$ & 0.9999 \\
30 & 5\,590 & 0.436 & $1.56{\times}10^{-1}$ & $6.17{\times}10^{-3}$ & $3.99{\times}10^{-1}$ & 0.9998 \\
50 & 6\,440 & 0.433 & $2.76{\times}10^{-1}$ & $1.94{\times}10^{-2}$ & $6.54{\times}10^{-1}$ & 0.9991 \\
\bottomrule
\end{tabular}
\end{table}

\paragraph{Covariance signal and ANOVA space}
Table~\ref{tab:as-signal-ablation} separates covariance choice from the effect of adding the ANOVA space: exact and source directions provide reference subspaces, while predictor rows use $\nabla\widetilde u$.
With the predictor direction, ANOVA-plus-global reaches $5.72\times10^{-5}$ versus $2.71\times10^{-3}$ for the $30$-feature global-only reference; the exact- and source-direction global-only references each reach $8.73\times10^{-8}$.
Because the compared spaces have $30$ and $5430$ unknowns, this table provides complementary signal/space diagnostics; Table~\ref{tab:random-oblique} supplies the near-width ablation.

\begin{table}[!htbp]
\centering
\caption{Covariance signal and approximation space on the $d=30$ oblique Poisson problem. Global-only AS-RFM uses $30$ features; HA-RFM adds the $K=2$ ANOVA space. Predictor-row CPU includes covariance formation, while unknowns count the final fitted space.}
\label{tab:as-signal-ablation}
\scriptsize
\setlength{\tabcolsep}{4pt}
\begin{tabular}{llrrrr}
\toprule
\textbf{method} & \textbf{covariance/space} & \textbf{unknowns} & \textbf{rel.\ $L^2$} & \textbf{alignment} & \textbf{CPU (s)} \\
\midrule
AS-RFM & exact dir. / global only & 30 & $8.73{\times}10^{-8}$ & $1.000$ & $<0.1$ \\
AS-RFM & source cov. / global only & 30 & $8.73{\times}10^{-8}$ & $1.000$ & $<0.1$ \\
AS-RFM & predictor / global only & 30 & $2.71{\times}10^{-3}$ & $>0.9999$ & 75.9 \\
HA-RFM & source cov. / ANOVA+global & 5\,430 & $1.49{\times}10^{-7}$ & $1.000$ & 121 \\
HA-RFM & predictor / ANOVA+global & 5\,430 & $5.72{\times}10^{-5}$ & $>0.9999$ & 151 \\
\bottomrule
\end{tabular}
\end{table}

\paragraph{Baseline comparisons}
\label{subsec:exp-baselines}
Figure~\ref{fig:accuracy-efficiency} compares the oblique Poisson target with equal-width full RFM and a fixed-configuration PINN.
Its source-derived HA-RFM direction isolates trial-space allocation; Tables~\ref{tab:predictor-as} and~\ref{tab:random-oblique} test identification from $\nabla\widetilde u$.

\begin{figure}[!htbp]
\centering
\includegraphics[width=0.96\textwidth]{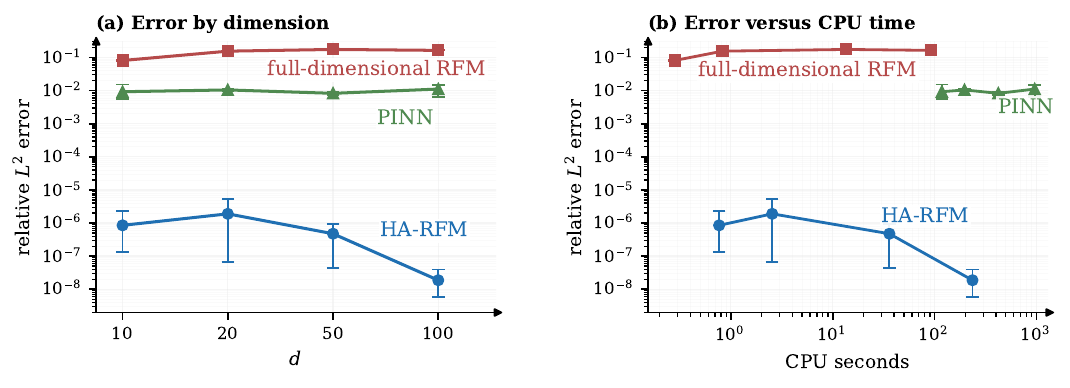}
\caption{Accuracy and cost on the oblique Poisson problem. HA-RFM uses a source-covariance direction for the equal-width comparison; Tables~\ref{tab:predictor-as} and~\ref{tab:random-oblique} report predictor-gradient alignment. RFM markers show three-draw means and min--max ranges.}
\label{fig:accuracy-efficiency}
\end{figure}

\paragraph{Fokker--Planck comparison}
\label{subsec:exp-fp}
On $\Omega = (-1,1)^d$ we solve
\[
L \rho := -\Delta \rho - \nabla \rho \cdot \nabla U(x) - \rho\, \Delta U(x) = f,
\qquad
\rho|_{\partial\Omega}=g,
\]
where $\kappa_{\rm FP}:=0.05$ is the Fokker--Planck coupling parameter,
$U(x) = \tfrac{1}{2} \sum_i x_i^2 + \kappa_{\rm FP} \sum_{i<j} x_i x_j$, and
$\rho_{\rm exact}(x) = 1 + 0.5 \cos(s)$ with $s = d^{-1}\sum_i x_i$.
Both the solution and manufactured source share the rank-one coordinate $s$, which is identified from the source covariance formed with $\nabla f$.
Indeed, applying the operator to $\rho_{\rm exact}$ produces a source depending on the same average coordinate.
The final solve uses the $K=1$ ANOVA space and $M_{\rm glob}=30$ global features with $R_{\rm glob}=5/\sqrt d$.
Once this coordinate is identified, the global features avoid the $\binom{100}{2}=4{,}950$ pair blocks of a coordinate-aligned $K=2$ space at $d=100$.
The method reaches relative $L^2$ error at most $5 \times 10^{-8}$ at every $d$ and $3.67 \times 10^{-10}$ at $d = 100$ in 89 s (Table~\ref{tab:exp2}).

\begin{table}[H]
\centering
\caption{Weakly coupled Fokker--Planck test with the active-subspace covariance formed from $\nabla f$; the source and exact solution share the same rank-one average coordinate.}
\label{tab:exp2}
\small
\setlength{\tabcolsep}{5pt}
\begin{tabular}{ccccccc}
\toprule
\textbf{$d$} & \textbf{$K$} & \textbf{$M_1$} & \textbf{$M_{\rm glob}$} & \textbf{signal} & \textbf{rel. $L^2$} & \textbf{CPU (s)} \\
\midrule
10  & 1 & 80 & 30 & source cov. & $2.40 \times 10^{-8}$  & 1.5 \\
20  & 1 & 80 & 30 & source cov. & $2.52 \times 10^{-8}$  & 3.9 \\
50  & 1 & 80 & 30 & source cov. & $1.32 \times 10^{-10}$ & 41 \\
100 & 1 & 50 & 30 & source cov. & $3.67 \times 10^{-10}$ & 89 \\
\bottomrule
\end{tabular}
\end{table}

\subsection{Semilinear equations and dense Riccati coupling}
\label{subsec:exp-picard}

HA-RFM solves each successive linearization; Theorem~\ref{thm:picard-inexact} propagates the inner error for contractive exact updates.
The isotropic HJB, Allen--Cahn, and dense Riccati runs use at most six, eight, and four steps.
They stop when the root-mean-square iterate difference divided by the new-iterate root-mean-square magnitude falls below $10^{-10}$ on a fixed $20\,000$-point Monte Carlo set independent of collocation and screening.

Set the drift, state-cost, and discount parameters to
$\alpha_{\rm drift}=\beta_{\rm HJB}=\gamma_{\rm HJB}=1$ and let $A=-\alpha_{\rm drift}I$.
The first family is the discounted stationary LQR Hamilton--Jacobi--Bellman equation
\[
\gamma_{\rm HJB} V - (Ax)\cdot\nabla V + \tfrac14|\nabla V|^2 - \tfrac12\Delta V
= \beta_{\rm HJB} |x|^2
\quad\text{on }(-1,1)^d.
\]
It has the exact quadratic solution
$V_{\rm exact}(x)=c\sum_i x_i^2+cd/\gamma_{\rm HJB}$, where
$c^2+(\gamma_{\rm HJB}+2\alpha_{\rm drift})c-\beta_{\rm HJB}=0$ and $c\approx0.30278$.
At Picard step $\ell+1$ we replace
$|\nabla V^{\ell+1}|^2$ by $-|\nabla V^\ell|^2+2\nabla V^\ell\cdot\nabla V^{\ell+1}$, giving a linear advection-reaction-diffusion equation for $V^{\ell+1}$.

The second family is the Allen--Cahn equation
$-\Delta u+u^3-u=f$ on $(-1,1)^d$, with
$u_{\rm exact}(x)=\tanh(s)$ and $s=d^{-1}\sum_i x_i$.
The Picard update
$(u^{\ell+1})^3\approx -2(u^\ell)^3+3(u^\ell)^2u^{\ell+1}$
again leaves a linear PDE for the next iterate, now with reaction $3(u^\ell)^2-1$ and source $f+2(u^\ell)^3$.
LQR is coordinate-additive, whereas Allen--Cahn uses source-covariance features along its rank-one oblique coordinate with $R_{\rm glob}=5/\sqrt d$ and reaches $4.39\times10^{-9}$ relative $L^2$ error at $d=100$ (Table~\ref{tab:exp-picard}).

\begin{table}[!htbp]
\centering
\caption{Semilinear elliptic problems solved by Picard iterations with HA-RFM inner solves. HJB preserves coordinate additivity; Allen--Cahn uses a source-covariance rank-one oblique coordinate.}
\label{tab:exp-picard}
\small
\setlength{\tabcolsep}{5pt}
\begin{tabular}{cccccccc}
\toprule
\multicolumn{4}{c}{\textbf{LQR HJB}} & \multicolumn{4}{c}{\textbf{Allen--Cahn}} \\
\cmidrule(lr){1-4}\cmidrule(lr){5-8}
\textbf{$d$} & \textbf{$M_1$} & \textbf{rel. $L^2$} & \textbf{CPU (s)} & \textbf{$d$} & \textbf{$M_1, M_{\rm glob}$} & \textbf{rel. $L^2$} & \textbf{CPU (s)} \\
\midrule
4  & 60 & $2.55 \times 10^{-9}$ & 4.5  & 10  & 60, 30 & $3.57 \times 10^{-8}$ & 22 \\
10 & 60 & $2.36 \times 10^{-9}$ & 6.2  & 20  & 60, 30 & $4.51 \times 10^{-7}$ & 60 \\
20 & 50 & $2.10 \times 10^{-9}$ & 15.4 & 50  & 60, 30 & $7.99 \times 10^{-9}$ & 433 \\
50 & 50 & $1.57 \times 10^{-9}$ & 146  & 100 & 50, 30 & $4.39 \times 10^{-9}$ & 2392 \\
\bottomrule
\end{tabular}
\end{table}

\paragraph{Dense Riccati case}
\label{subsec:exp-aniso-hjb}
To test dense coupling, we retain the LQR HJB structure but use the random symmetric negative-definite drift
$A = -\alpha_{\rm drift} I - \varrho_{\rm ric}G_{\rm ric}G_{\rm ric}^\top$,
where $G_{\rm ric}\in\mathbb R^{d\times d}$ has i.i.d.\ $\mathcal{N}(0,1/d)$ entries, with $\alpha_{\rm drift}=1$ and $\varrho_{\rm ric}=0.5$.
Taking the control and control-cost matrices $\mathbf B_{\rm ctrl}$ and $\mathbf R_{\rm ctrl}$ to be $I_d$, the exact value is $V_{\rm exact}(x)=x^\top P x+\mathrm{tr}(P)/\gamma_{\rm HJB}$, where the symmetric matrix $P\in\mathbb R^{d\times d}$ solves
\[
\gamma_{\rm HJB} P - A^\top P - P A + P^2 = \beta_{\rm HJB} I_d .
\]
For $d\ge10$, the off-diagonal/diagonal Frobenius-norm ratio of $P$ is about $0.18$, while $\nu_1/\nu_2=1.00$--$1.02$ provides no low-rank signal; the reported space therefore contains all order-two blocks and no active augmentation.
Using the isotropic-LQR update, the iteration reaches its plateau within two steps.
Figure~\ref{fig:secondary-mechanisms}(a) shows the fixed-order block growth described by Theorem~\ref{thm:complexity}: from $d=5$ to $30$, the block count increases from $15$ to $465$; under the reported per-block widths, the error changes from $1.8\times10^{-6}$ to $1.2\times10^{-3}$ and solve time from $5$ to $1511$ seconds.

\subsection{Non-separable coefficients}
\label{subsec:exp-var-coef}
To test coefficient-induced coupling on $\Omega=(0,1)^d$, we use $\kappa(x)=1+0.5\,\sin(\pi(x_1+x_2))$ and the order-two exact solution
$u_{\rm exact}=u_1+u_2$, where $u_1(x)=\sum_i\sin(\pi x_i)$ and $u_2(x)=\tfrac12\sum_{i<j}\sin(\pi x_i)\sin(\pi x_j)$.
The matching source and boundary data give superposition dimension $K_0=2$ (that is, $K_\tau$ at $\tau=0$), while $\kappa$ couples coordinates 1 and 2 at the operator level.
Specifically, $f=-\Delta u+\kappa u$ and the Dirichlet datum is the trace of $u_{\rm exact}$.

Figure~\ref{fig:secondary-mechanisms}(b) compares HA-RFM at $K \in \{1, 2\}$.
At $d=10$, the order-1 truncation stagnates at $2.0\times10^{-2}$, while all $45$ pair blocks give $3.1\times10^{-4}$; residual-Sobol screening retains $42$ of the $45$ pair blocks, including the coefficient-coupled pair $(1,2)$, and reaches $5.9\times10^{-3}$.
At $d=20$, all pairs improve the $K=1$ error by roughly one order of magnitude, while capped Sobol screening retains $80$ of the $190$ pairs and gives $1.3\times10^{-2}$.
At $d=50$, retaining $120$ of the $1225$ pairs leaves the error at the $K=1$ level, showing that distributed pair structure requires broader retention as $d$ increases.

\begin{figure}[!htbp]
\centering
\includegraphics[width=0.96\textwidth]{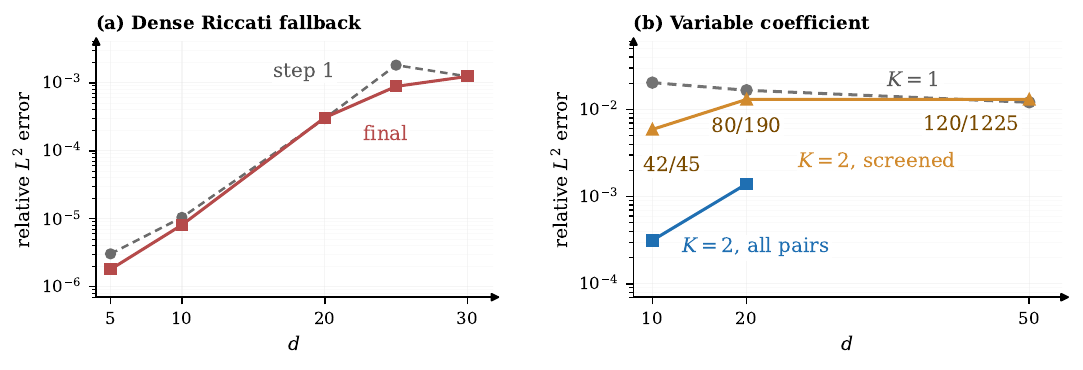}
\caption{Two cases governed by coordinate interactions. Dense-coupled Riccati HJB uses the order-two ANOVA route; for the variable-coefficient problem, all-pairs $K=2$ realizes the interaction-space gain, while capped screening identifies when broader retention is required.}
\label{fig:secondary-mechanisms}
\end{figure}

\subsection{Stability and structure-selection validation}
\label{subsec:exp-mc}

\paragraph{Independent validation and stability}
Without using the least-squares residual, the $d=10$ HA-RFM Picard solution matches the analytic Riccati value to root-mean-square relative error $1.3\times10^{-9}$ at six closed-loop points; Monte Carlo cost evaluation agrees to $2.2\times10^{-3}$, consistent with sampling error.
The three-draw ranges in Figure~\ref{fig:accuracy-efficiency} preserve the RFM accuracy ordering.
Across three feature seeds, the $d=10$ predictor-gradient oblique test gives relative $L^2$ errors of $9.47\times10^{-7}$--$1.34\times10^{-6}$, with direction alignment exceeding $0.9999$ in every run, while the dense Riccati errors are $7.60\times10^{-6}$--$8.09\times10^{-6}$.
With the feature seed fixed, increasing $(N_{\rm int},N_{\rm bc})$ from $(2500,625)$ to $(4000,1000)$ and $(8000,2000)$ decreases the Riccati error monotonically from $1.29\times10^{-5}$ to $8.09\times10^{-6}$ and $5.79\times10^{-6}$.

\paragraph{Spectral and residual selection}
\label{subsec:exp-ablation}
\label{subsec:sobol-validation}

With $\gamma_{\rm gap}=10$, the measured eigengap ratios place the oblique Poisson predictor covariance and Fokker--Planck source covariance in the low-rank regime, whereas dense Riccati has $\nu_1/\nu_2\approx1$ and supports no active augmentation; Table~\ref{tab:as-signal-ablation} confirms an accurate joint solve with the predictor-estimated direction.
For residual-Sobol recovery, we use a $d=15$ Poisson solution containing only pairs $(1,2)$, $(3,4)$, and $(5,6)$ among $105$ candidates.
A singleton predictor with $M_1=30$ is fitted using $(N_{\rm int},N_{\rm bc})=(6000,1500)$, after which the rule~\eqref{eq:active-set} with $\vartheta_{\rm Sob}=0.05$ is applied to the actual PDE residual $f-L\widehat u_1$.
Across the nine predictor-feature/QMC seed combinations in Table~\ref{tab:residual-support-recovery}, every run retains all true pairs; at $M_{\rm QMC}=4096$, all nine are exact with no false positives.
The independent pick-freeze result supplies the concentration theory, while these randomized-QMC estimates directly test practical recovery of the same closed indices.

\begin{table}[!htbp]
\centering
\caption{Residual-Sobol recovery of three active pairs among $105$ candidates at $d=15$, summarized over nine predictor-feature/QMC-seed combinations. Precision and recall use the exact three-pair support.}
\label{tab:residual-support-recovery}
\small
\setlength{\tabcolsep}{6pt}
\begin{tabular}{ccccc}
\toprule
\textbf{$M_{\rm QMC}$} & \textbf{exact recovery} & \textbf{mean precision} & \textbf{recall} & \textbf{selected-pair range} \\
\midrule
2\,048 & $3/9$ & 0.635 & 1.000 & 3--15 \\
4\,096 & $9/9$ & 1.000 & 1.000 & 3--3 \\
\bottomrule
\end{tabular}
\end{table}

The distributed-interaction runs show that broader pair mass requires a larger candidate cap, with post-selection accuracy also governed by the per-subset width $M_k$.
Across these tests, residual Sobol scores identify coordinate interactions, whereas gradient covariance identifies oblique active coordinates.

\section{Conclusion}
\label{sec:conclusion}

We developed HA-RFM, a PDE-driven random-feature framework that converts residual Sobol structure and predictor-gradient covariance into coordinate-aligned and oblique trial spaces, then fits them jointly by regularized least squares.
The analysis connects structural truncation, finite-width approximation, and sampled fitting to the final $L^2$ error, while the experiments show that the detected structure is translated into accurate approximations across coordinate-aligned, oblique, and semilinear regimes.
The resulting structure-adaptive trial spaces provide a practical route to high-dimensional elliptic PDE approximation.

Promising extensions include adaptive searches over larger interaction families and iterative updates of the coordinate and oblique blocks during nonlinear solves.
Extending the framework beyond product reference measures and elliptic operators would further open applications to correlated inputs and time-dependent high-dimensional PDEs.

\bibliographystyle{siamplain}
\bibliography{refs}

\end{document}